\documentclass{article}
\usepackage[a4paper,top=2.54cm,bottom=2.54cm,left=3.17cm,right=3.17cm]{geometry}
\usepackage[shortlabels]{enumitem}
\setenumerate[1]{itemsep=0pt,partopsep=0pt,parsep=\parskip,topsep=0pt}
\usepackage[tworuled,linesnumbered,noline,noend]{algorithm2e}
\usepackage{amsmath,amsthm,float,amssymb,mathtools}
\usepackage[hidelinks]{hyperref}
\usepackage[numbers,sort&compress]{natbib}
\usepackage[rgb,dvipsnames]{xcolor}
\usepackage{tikz}
\usepackage{mathrsfs}
\usepackage{array}

\allowdisplaybreaks

\newtheorem{theorem}{Theorem}[section]
\newtheorem{lemma}[theorem]{Lemma}

\newtheorem{corollary}[theorem]{Corollary}

\newtheorem{proposition}[theorem]{Proposition}

\newcommand{\opt}{\mathrm{OPT}}
\newcommand{\bv}{\bar{V}}

\newcommand{\tk}{\texorpdfstring{$k$}{k}}

\newcommand{\tkm}{\texorpdfstring{$\underline{k}$}{k-}}

\newcommand{\argmax}{\mathrm{argmax}}

\newcommand{\email}[1]{\href{mailto:#1}{\nolinkurl{#1}}}

\SetKw{KwDownto}{downto}

\newcounter{algprocedure}
\newcounter{savedalgocf}

\newenvironment{algprocedure}[1][htbp]{%
  \setcounter{savedalgocf}{\value{algocf}}%
  \setcounter{algocf}{\value{algprocedure}}%
  \stepcounter{algocf}%
  \SetAlgoRefName{\relax\arabic{algocf}}%
  \begin{algorithm}[#1]
}{%
  \end{algorithm}
  \setcounter{algprocedure}{\value{algocf}}%
  \setcounter{algocf}{\value{savedalgocf}}%
}
  
\begin{document}
\title{\textbf{Supermodular Maximization with\\ Cardinality Constraints}\thanks{Research supported in part by the National Natural Science Foundation of China (NSFC)  under grant numbers 11971046, 12331014 and 71871009; Science and Technology Commission of Shanghai Municipality (No. 22DZ2229014)}}
\author{Xujin Chen\ \ \ \  Xiaodong Hu\ \ \ \  Changjun Wang\\
\footnotesize Academy of Mathematics and Systems Science, Chinese Academy of Sciences, Beijing 100190, China\\
{\footnotesize School of Mathematical Sciences. University of Chinese Academy of Sciences, Beijing 100049, China}\\
\texttt{\small\{xchen,xdhu, wcj\}@amss.ac.cn} 
\and
Qingjie Ye\\
{\footnotesize School of Mathematical Sciences, Key Laboratory of MEA (Ministry of Education),}\\
{\footnotesize Shanghai Key Laboratory of PMMP, East China Normal University, Shanghai, 200241, China}\\
\texttt{\small qjye@math.ecnu.edu.cn}}

\date{}
\maketitle

\begin{abstract}
    Let $V$ be a finite set of  $n$ elements, $f: 2^V \rightarrow \mathbb{R}_+$ be a nonnegative monotone supermodular function, and $k$ be a positive integer no greater than $n$. This paper addresses the problem of maximizing $f(S)$ over all subsets $S \subseteq V$ subject to the cardinality constraint $|S| = k$ or $|S|\le k$. 
    
   Let $r$ be a constant integer. The function $f$ is assumed to be {\em $r$-decomposable}, meaning there exist $m\,(\ge1)$ subsets $V_1, \dots, V_m$ of $V$, each with a cardinality at most $r$, and a corresponding set of nonnegative supermodular functions  $f_i : 2^{V_i} \rightarrow \mathbb{R}_+$, $i=1,\ldots,m$ such that $f(S) =\sum_{i=1}^m f_i(S \cap V_i)$ holds for each $S \subseteq V$. Given $r$ as an input, we present a polynomial-time $O(n^{(r-1)/2})$-approximation algorithm for this maximization problem, which does not require prior knowledge of the specific decomposition.
    
    When the decomposition $(V_i,f_i)_{i=1}^m$ is known, an additional connectivity requirement is introduced to the problem. Let $G$ be the graph with vertex set $V$ and edge set  $\cup_{i=1}^m \{uv:u,v\in V_i,u\neq v\}$. The cardinality constrained solution set $S$ is required to induce a connected subgraph in $G$. This model generalizes the well-known problem of finding the densest connected $k$-subgraph. We propose a polynomial time $O(n^{(r-1)/2})$-approximation algorithm for this generalization. Notably, this algorithm gives an $O(n^{1/2})$-approximation for the densest connected $k$-subgraph problem, improving upon the previous best-known approximation ratio of $O(n^{2/3})$.
\end{abstract}


\section{Introduction}\label{sec:intr}
Given a function  $f: 2^V \rightarrow \mathbb{R}$ over subsets of a finite ground set $V$, for any $v\in V$ and $S\subseteq V$, 
let $f(v|S) = f(S\cup \{v\})-f(S)$ denote the \emph{marginal value} of $v$ with respect to $S$. Function $f$ is {\em supermodular} if $f(A)+f(B) \le f(A\cup B)+f(A\cap B)$ for all $A,B \subseteq V$, or equivalently $f(v|B)\le f(v|A)$ for all $B\subseteq A$ and $v\in V \setminus A$. This paper adopts the value-oracle model: given any $S\subseteq V$, a black-box oracle returns the value $f(S)$.

Supermodular functions play a fundamental role in both theoretical studies and practical applications. Their theoretical significance is evident in several well-established results, such as the fact that a set function is supermodular if and only if its Lov\'{a}sz extension is concave \citep{lovasz1983submodular}, and that any set function can be represented as the difference between two supermodular functions \citep{narasimhan2005submodular}.
Moreover, supermodular functions have found extensive applications across diverse fields. For example, they are used in the analysis of economic systems and game theory to model increasing returns and strategic complementarity \citep{Topkis1998}, and in the design of algorithms for optimization problems such as scheduling \citep{chamon2022approximately}, resource allocation \citep{paccagnan2022utility}, and supply chain management \citep{iancu2013supermodularity}.

\paragraph{Cardinality constraints.}
Many real-world applications involve maximization over a  %{monotone (non-decreasing)} 
nonnegative supermodular function $f: 2^V \rightarrow \mathbb{R}_+$ subject to a cardinality constraint. Examples include public-key encryption, computational biology, data sampling, and document summarization (see, e.g., \cite{abw10,abdhk14,bai18greed}). 
Suppose that the ground set $V$ comprises $n$ elements, and  $k$ is a positive integer not exceeding $n$. A subset $S$ of $V$ is classified as: a \emph{$k$-subset} if $|S|=k$, a \emph{$\underline{k}$-subset} if $|S|\le k$, and a \emph{$\overline{k}$-subset} if   $|S|\ge k$. The constrained maximization typically involves finding either a $k$-subset or a $\underline{k}$-subset $S$ of $V$ that maximizes $f(S)$. The problem is referred to as the \emph{${k}$-subset supermodular maximization} (\emph{${k}$SPM}) problem in the former case, and as the \emph{$\underline{k}$SPM} problem in the latter case.

If $f$ is monotone, i.e., $f(B)\le f(A)$ for all $B\subseteq A$, then the $\underline{k}$SPM problem on   $f$ is equivalent to the $k$SPM problem, in that an optimal solution of one problem can be converted to an optimal solution of the other with the same objective value. 

\paragraph{Decomposability.} In practice, the supermodular function $f$ often exhibits certain structural properties related to decomposition, which helps to obtain better approximations for supermodular maximization \cite{chekuri2022densest}. Let $r$ be a positive integer.
A supermodular function $f: 2^V \rightarrow \mathbb{R}_+$ is {\em $r$-decomposable} if there exist subsets $V_1, \dots, V_m$ of $V$ and  supermodular functions $f_i : 2^{V_i} \rightarrow \mathbb{R}_+$, $i=1,\ldots,m$ such that $|V_i| \le r$ for each $i$ and $f(S) =\sum_{i=1}^m f_i(S \cap V_i)$ for each $S \subseteq V$. We call the corresponding decomposition a \emph{$r$-decomposition}, and denote it as $(r;(V_i,f_i)_{i=1}^m)$, or $(V_i,f_i)_{i=1}^m$ if $r$ is clear from the context. 

Trivially, $r$ could be taken as $n$, for which the decomposition is $(V,f)$. Generally, the smaller $r$ is, the better the structural properties $f$ has, and the more algorithmically amenable the maximization on $f$ becomes. As a key example, it is not hard to prove that $f$ is $1$-decomposable if and only if $f$ is modular; that is, $f(A)+f(B)=f(A\cup B)+f(A\cap B)$ for all $A,B\subseteq V$. For modular functions, the \tk SPM and $\underline{k}$SPM problems simplify to selecting the top $k$ most valuable elements.

\paragraph{Connectivity constraints.} Given the decomposition $(V_i,f_i)_{i=1}^m$ for supermodular function $f$, we construct an undirected simple graph $G=(V,E)$ with vertex set $V$ and edge set $E:=\cup_{i=1}^m \{uv:\{u,v\}\subseteq V_i,u\neq v\}$. Every nonempty subset $S$ of $V$ {\em induces} a subgraph of $G$, denoted by $G[S]$, whose vertex set is $S$ and whose edge set consists of edges in $E$ with both end-vertices in $S$. We call the set $S$ {\em connected} if the induced graph $G[S]$ is connected. Imposing connectivity requirements on subsets arises naturally in many applications (see, e.g., \cite{abdhk14,dkmp25,kss15}). This leads to supermodular optimization involving connectivity and cardinality constraints: the \emph{connected $k$-subset supermodular maximization} (C$k$SMP)  problem for finding a connected $k$-subset $S$ with maximum $f(S)$, and the \emph{connected $\underline{k}$-subset supermodular maximization} (C$\underline{k}$SMP)  problem for finding a connected $\underline{k}$-subset $S$ with maximum $f(S)$. 

Unlike the counterpart without connectivity requirement, even if $f$ is monotone, the cardinality constraint $|S|\le k$ in the C$\underline{k}$SPM 
 problem cannot be tightened to $|S|=k$ because possibly no $k$-subset is connected. For example, if $|V_i|=1$ for all $1\le i\le m$, then $G$ is edgeless, and no set of cardinality 2 or more is connected. On the other hand, for monotone functions, the C\tk SPM and C\tkm SPM problems are ``indirectly'' equivalent. This means that the algorithm for one problem can be used to solve the other by examining all connected components of $G$,  without any loss of solution quality or time complexity.

 \bigskip 
For convenience, throughout the paper, we use $\mathscr{F}_r$ to denote the family of monotone nonnegative supermodular functions that are $r$-decomposable, where \emph{$r$ is a constant}.  The main goal of this paper is to develop polynomial-time  $O(n^{(r-1)/2})$-approximation algorithms for the $k$SPM, \tkm SPM, and C\tkm SPM problems on $\mathscr{F}_r$. The algorithms for C\tkm SPM apply to C$k$SPM with an easy expansion from a $\underline{k}$-set to a $k$-set.

\subsection{Related work} 
Maximizing a supermodular function is equivalent to minimizing its negation, i.e., a submodular function. The unconstrained submodular minimization is solvable in (strongly) polynomial time (see \citep{grotschel1981ellipsoid,schrijver2000combinatorial}). Extensive research has been conducted on the optimization problems associated with both submodular and supermodular functions.   The reader is referred, e.g.\ to \citep{nemhauser1978analysis,svitkina2011submodular,buchbinder2015tight,feige2011maximizing,nemhauser1978best,calinescu2011maximizing} for various algorithmic studies in this area.

\citet{svitkina2011submodular} studied submodular minimization with cardinality constraints. They proved that any algorithm that makes a polynomial number of oracle queries cannot achieve an $o\left(\sqrt{\frac{n}{\ln n}}\right)$-approximation for finding a minimum-valued $\underline{k}$-set. 
On the other hand, they proposed a $(5\sqrt{\frac{n}{\ln n}},\frac 12)$-bicriteria approximation algorithm for $\overline{k}$-sets, which outputs a $\overline{k/2}$-set whose value is at most $5\sqrt{\frac{n}{\ln n}}$ times the minimum value of a $\overline{k}$-set. Besides, they showed that there is no $(\rho, \sigma)$-bicriteria approximation algorithm that makes a polynomial number of oracle queries, for any $\rho$ and $\sigma$ with $\frac \rho \sigma=o(\sqrt{\frac{n}{\ln n}})$.   

\citet{bai18greed} designed a $\beta_f$-approximation algorithm for the $k$SPM problem when $f$ is monotone, where $\beta_f=\max_{v\in V} f(v|V\setminus\{v\})/f(v)$. %If $f(v|V\setminus\{v\})=f(v)=0$, then  $f(v|V\setminus\{v\})/f(v)=1$ is assumed. 
Moreover, they %\citet{bai18greed} 
proved that for any $\epsilon>0$, the $k$SPM problem does not admit any  $(\beta_f-\epsilon)$-approximation in polynomial time.  
However, the ratio $\beta_f$ is unbounded in a general sense as there may be $v\in V$ such that $f(v)=0$ and $f(v|V\setminus\{v\})>0$ (see Appendix~\ref{sec:bai} for an example of $f\in\mathscr F_2$).
In addition, they proved that for any $\epsilon > 0$, there is no $O(2^{(0.5-\epsilon) n})$-approximation in polynomial time even when considering only monotone functions and allowing randomized algorithms.

To the best of our knowledge, apart from the work by \citet{bai18greed}, no approximation algorithms were proposed for the general $k$SPM and C$k$SPM problems, even under the common assumption that the input nonnegative supermodular function is monotone and 
$r$-decomposable.  In contrast, a substantial literature addresses special cases of these problems, mainly focusing on instances arising from hypergraphs and ordinary graphs. 

\paragraph{Special cases of $k$SPM.}
Let $H = (V,E)$ be a hypergraph with vertex set $V$ and hyperedge set $E$. For any subset $S$ of $ V$, let $E(S)=\{e\in E:e\subseteq S\}$ denote the set of hyperedges in the subhypergraph of $H$ induced by $S$. The \emph{density} of this subhypergraph, a.k.a.\ the \emph{density} of $S$, is defined as  $\mu(S)=|E(S)|/|S|$. While the problem of finding a densest subhypergraph is solvable in polynomial time \citep{chlamtac2018densest,hu2017maintaining}, its cardinality-constrained variant is NP-hard. This variant is known as the \emph{densest $k$-subhypergraph} (D$k$SH) problem, where one needs to find a $k$-subset $S$ of $V$ that maximizes the density $\mu(S)=|E(S)|/k$, or equivalently, maximizes the number $|E(S)| $ of hyperedges contained in $S$. It is obvious that $\mu(S)=|E(S)|/k$ is a supermodular function on $2^V$, and therefore, the $k$SPM problem generalizes the D$k$SH problem. Besides, the supermodular function $\mu$ is nonnegative, monotone, and $r$-decomposable, where $r=\max\{|e|:e\in E\}$ is often called  the \emph{rank} of $H$.
Although  $\beta_\mu=+\infty$ in general, researchers managed to obtain better approximations for uniform hypergraphs. A hypergraph is $r$-uniform if every hyperedge in it consists of exactly $r$ vertices. \citet{chlamtac2018densest} designed an $O(n^{0.697831+\epsilon})$-approximation algorithm for D$k$SH on 3-uniform hypergraphs that runs in $n^{O(1/\varepsilon)}$ time. However, this approximation does not extend to the general rank-$3$ hypergraphs.

When restricted to 2-uniform hypergraphs, i.e., ordinary graphs $G$, the D$k$SH problem reduces to the well-known \emph{densest $k$-subgraph} (D$k$S) problem, which was first investigated by \citet{kortsarz1993choosing} in 1993. Since then, it has been extensively studied in both directions of designing approximation algorithms and proving inapproximation lower bounds \citep{kortsarz1993choosing,feige2001dense,feige2001approximation,asahiro2000greedily,bhaskara2010detecting,bhaskara2012polynomial,braverman2017eth,manurangsi2017almost}.

On the positive side, as shown by \citet{feige2001dense}, every approximation algorithm designed for the unweighted version of D$k$S, in which all edges have weight 1, can be extended to handle the (general) weighted case. This extension introduces an additional $O(\log n)$ factor to the approximation ratio. The approximation ratio for weighted D$k$S has been improved over the years. \citet{kortsarz1993choosing} first provided {an $\mathcal O(n^{0.3885})$-approximation algorithm {that} runs in  $\mathcal O(n^5)$ time, where the notation $\mathcal O$ ignores polylogarithmic factors in $n$}. Later, \citet{feige2001dense} improved the approximation ratio to $O(n^{1/3-\delta})$ for some small $\delta>0$ (which was estimated to be roughly 1/60). This approximation ratio had remained the best for almost ten years until 2010 when \citet{bhaskara2010detecting} proposed an $O(n^{1/4+\varepsilon})$-approximation algorithm that runs in $n^{O(1/\varepsilon)}$ time. This remains the state-of-the-art approximation ratio for the D$k$S problem when parameterized solely by $n$. When the desired subgraph size $k$ is also involved, combining several ideas from binary search, ``expanding kernels'' and conditional probability, \citet{kortsarz1993choosing} gave an $O(n/k)$-approximation algorithm {that} runs in $O(m\log W)$ time, where $m$ is the number of edges in graph $G$ and $W$ is the total weight of $\min\{m,k^2\}$ edges of highest weights. When $k<n/3$, \citet{asahiro2000greedily} proved the approximation ratio of $\Theta(n/k)$ for the heuristic that repeatedly removes a vertex with the minimum {weighted degree} until there are $k$ vertices left. Later, \citet{feige2001dense} showed that a different procedure using greedy attachments finds an {$O(n/k)$-approximation}  in $O(m+n\log n)$ time. Although linear and semidefinite programming relaxation approaches have been adopted in \citep{feige2001approximation,han2002improved,srivastav1998finding} to design algorithms with approximation {ratios} somewhat better than $O(n/k)$ for certain value ranges of $k$, the $O(n/k)$ ratio remains the state-of-the-art for D$k$S in the general case, applying to all values of $n$ and $k$ for both weighted and unweighted graphs. For some special cases characterized by restrictive graph classes, specific values of parameter $k$, and particular (weighted) densities of optimal solutions, improved approximations have been achieved {\citep{arora1999polynomial,kortsarz1993choosing,demaine2005algorithmic,liazi2007densest,liazi2008constant,chen2011densest,nonner2016ptas}}. 

On the negative side, no nontrivial lower bound on the approximability of D$k$S has been known under the standard assumption of $\text{P}\neq \text{NP}$. However, $(1+\varepsilon)$-inapproximability  \cite{feige2002relations,khot2006ruling},  constant-factor-inapproximability and superconstant-factor-inapproximability \cite{raghavendra2010graph,aammw11} were established for (unweighted) D$k$S under various complexity assumptions.  \citet{manurangsi2017almost} showed that, under the exponential time hypothesis (ETH), no polynomial-time algorithm can approximate D$k$S to within $n^{1/(\log\log n)^c}$ factor of the optimum, where $c>0$ is a universal constant independent of $n$. Recently, \citet{jones2023sum} provided the hardness evidence for D$k$S by showing lower bounds against the powerful Sum-of-Squares (SoS) algorithm: for any $\epsilon > 0$, there exists a constant $\delta > 0$ such that degree-$n^{\delta}$ SoS exhibits an integrality gap of $O(n^{1/4 - \epsilon})$ for D$k$S. The gap matches the approximation ratio by \citet{bhaskara2010detecting} mentioned above.

\paragraph{Special case of C${k}$SPM.} When graphical connectivity constraint is imposed on the D$k$S, the problem transforms to a special case of  C${k}$SPM. Given a graph $G=(V,E)$ with edge weights $w\in\mathbb R_+^E$, this yields a 2-decomposition $(\{u,v\},w_{uv}/k)_{uv\in E}$ of monotone supermodular function $\mu:2^V\rightarrow\mathbb R_+$ with $\mu(\{u,v\})=w_{uv}$ for all $uv\in E$. A subgraph with $k$ vertices in $G$ is termed a \emph{$k$-subgraph}. Assuming $G$ is connected, the \emph{densest connected $k$-subgraph} (DC$k$S) problem seeks a connected $k$-subgraph  in $G$ of maximum weighted density. This amounts to identifying a connected $k$-subset $S$ of $V$ with maximum $\mu(S)$. 

Compared with the large literature on D$k$S, the work on approximating the densest connected $k$-subgraphs is relatively limited. A trivial $O(k)$-approximation can be achieved by examining all stars \citep{kortsarz1993choosing}. \citet{chen2017finding} designed an $O(nm)$-time algorithm that finds a connected $k$-subgraph whose weighted density is at least  $\Omega(k^2/n^2)$ of the maximum weighted density among all $k$-subgraphs of $G$, which are not necessarily connected. It follows that DC$k$S is approximable within a factor of $O(\min\{k,n^2/k^2\})=O(n^{2/3})$. For the unweighted version of DC$k$S, an $O(n^{2/5})$-approximation can be guaranteed by combining the $O(n^2/k^2)$-approximation algorithm with various methods that search for dense connected $k$-subgraphs based on densest subgraphs, high-degree vertices and 2-hop neighborhoods of vertices \citep{chen2017finding}. Other works on {designing polynomial-time algorithms} for DC$k$S only deal  with special graphical topologies, such as trees \cite{ps83,cp84}, $h$-trees, cographs, split graphs \cite{cp84}, complete graphs \cite{ravi1994heuristic,hrt97}, and interval graphs whose clique graphs are paths~\cite{lmz05}.

\paragraph{Upper bound of ${k}$SPM.}\; The  ${k}$SPM problem over a supermodular function  $f: 2^V \rightarrow \mathbb{R}_+$ is related to the \emph{densest at-least-$k$ supermodular subset} (D$\overline{k}$SS) problem. The objective of D$\overline{k}$SS to find a $\overline{k}$-set $S$ of $V$ that maximizes $f(S)/|S|$. The optimal (approximate) solutions to D$\overline{k}$SS provide upper bounds on the objective values of the \tk SPM problem. Indeed, for any optimal solution $S^*$ of the D$\overline{k}$SS and any $k$-subset $S$, it holds that
$f(S^*)\ge f(S^*)\cdot \frac{|S|}{|S^*|}\ge \frac{f(S)}{|S|}\cdot |S|=f(S)$.
 \citet{chekuri2022densest} established the following result.

\begin{theorem}[\citet{chekuri2022densest}]\label{thm:dkpss}
    There is an $O(n^{2})$ time $(r+1)$-approximation algorithm for D$\overline{k}$SS on nonnegative supermodular functions that are monotone and $r$-decomposable. 
\end{theorem}

Moreover,  \citet{chekuri2022densest} designed a polynomial-time $2$-approximation algorithm for D$\overline{k}$SS on monotone nonnegative supermodular functions.

As a special case of D$\overline{k}$SS, the densest at-least-$k$ subgraph (D$\overline{k}$S) problem is to find a densest subgraph with $k$ or more vertices. \citet{khuller2009finding} proved that the D$\overline{k}$S problem is NP-hard, and proposed a maximum-flow-based  2-approximation algorithm for D$\overline{k}$S, which runs in $O(n^2m\log^2n)$ time. They also presented an LP-based polynomial-time algorithm with the same approximation ratio 2. The algorithm by \citet{andersen2009finding} attains an approximation ratio of 3 for D$\overline{k}$S with shorter running time $O(m+n\log n)$.

\subsection{Our results} 

In this paper, we propose $O(n^{(r-1)/2})$-approximation algorithms for supermodular maximization on $\mathscr{F}_r$ with a cardinality constraint or with both cardinality and connectivity constraints. The algorithm for $k$SPM and \tkm SPM runs in $O(n^{r+1})$ time, and that for C$k$SPM and C$\underline{k}$SPM runs in $O(mn^{r+1})$ time. 
These general settings of constrained supermodular maximization have been less explored in the literature. To the best of our knowledge, the only existing result is the $\max\{f(v|V\setminus\{v\})/f(v):v\in V\}$-approximation for $k$SPM \cite{bai18greed}, where the approximation ratio can be arbitrarily large, even for $f\in\mathscr F_2$, and when the size of the ground set is fixed at $n\ge4$ (see Appendix~\ref{sec:bai}).

Notably, our result on C${k}$SPM implies an $O(n^{0.5})$-approximation algorithm for the densest connected $k$-subgraph (DC$k$S) problem, improving the ratio in \citep{chen2017finding} by a multiplicative factor of $n^{0.1666}$. 

Because $r$ is a constant, our algorithm design focuses on the instances in which $k>r$.  Conversely, when $k \le  r$, the optimal solution can be found in $O(n^r)$ time through a brute-force search. Specifically, we develop four algorithms as listed in Table \ref{tb:algorithm}. Choosing the better solution from the outputs of the first (resp.\ last) two algorithms gives an $O(n^{(r-1)/2})$-approximation for $k$SPM and \tkm SPM (resp.\ C$\underline{k}$SPM) on $\mathscr F_r$. The solution of C$\underline{k}$SPM is easily expanded to a solution of C$k$SPM with the same performance guarantee.

\begin{table}[h!]
\centering
{\small
\renewcommand{\arraystretch}{1.1} 
\begin{tabular}{|c|c|c|p{5.5cm}|} % Set last column to left-aligned paragraph
\hline
\textbf{Problem} & \textbf{Approximation ratio} & \textbf{Runtime} & \multicolumn{1}{c|}{\textbf{Algorithm}} \\ % Center the last header
\hline
$k$SPM (\tkm SPM) & $O\left({n^{r-1}}/{k^{r-1}}\right)$ & $O(n^2)$ & Alg.\ \ref{alg:pmkss} (Greedy Peeling) \\
\hline
$k$SPM (\tkm SPM) & $O(k^{r-1})$ & $O(n^{r+1})$ & Alg.\ \ref{alg:chmkss2} (Batch-greedy Augmenting) \\
\hline
C$\underline{k}$SPM & $O(k^{r-1})$ & $O(mn^{r+1})$ & Alg.\ \ref{alg:star} (Star Exploration) \\
\hline
C$\underline{k}$SPM & $O\left({n^{r-1}}/{k^{r-1}}\right)$ & $O(mn^2)$ & Alg.\ \ref{alg:overall} (Removing-Attaching-Merging) \\
\hline
\end{tabular}
}
 \caption{Approximation algorithms for $k$SPM, \tkm SPM and C\tkm SPM problems on $\mathscr F_r$, where $k>r$}
    \label{tb:algorithm}
\end{table}

The algorithms listed in Table~\ref{tb:algorithm} differ in their input requirements. Algorithm \ref{alg:pmkss} requires the least information, needing only an oracle for the input function $f$. Algorithm \ref{alg:chmkss2} requires an oracle for $f$ and the value of $r$. In contrast, Algorithms \ref{alg:star} and \ref{alg:overall} are the most demanding, requiring the detailed decomposition information $(V_i, f_i)_{i=1}^m$ of the function $f \in \mathscr{F}_r$, including oracles for each component function $f_i$, $i=1,\ldots,m$.

We design Algorithms \ref{alg:pmkss} and \ref{alg:chmkss2} not only to solve the $k$SPM problem but also to serve as subroutines for Algorithms \ref{alg:star} and \ref{alg:overall}, respectively, to aid in solving the C\tkm SPM problem.

 The approximation ratios of our algorithms for C\tkm SPM are actually evaluated against the optimum objective value of the \tkm SPM problem (see Corollary \ref{cor:strong}). The stronger guarantees imply that the optimal solution without the connectivity constraint can be approximated by a connected solution within the stated ratio, though at the expense of a runtime increase by a factor of $m$.

\paragraph{Algorithms for \tk SPM and \tkm SPM.} We adopt two complementary strategies to construct a $k$-subset that approximates the optimal solution of the $k$SPM problem on $\mathscr F_r$. The first method, \emph{Greedy Peeling} (Algorithm \ref{alg:pmkss}), starts with the entire set of elements and iteratively removes the one with the smallest marginal contribution until a set of size $k$ remains. This approach yields an approximation ratio of $O(n^{r-1}/k^{r-1})$ (Theorem~\ref{thm:peeling}).

In contrast, the second method, \emph{Batch-greedy Augmenting} (Algorithm \ref{alg:chmkss2}), builds a solution from an empty set. Instead of adding a single element one by one, our algorithm employs a batch-greedy strategy, adding an $r$-subset that provides the maximum increase in function value at each step. This algorithm achieves an $O(k^{r-1})$ approximation (Theorem \ref{thm:add}). Note that the single-element-greedy augmenting by \citet{bai18greed} does not work well even for monotone 2-decomposable supermodular functions, as it suffers from an arbitrarily large approximation ratio independent of $n \ge 4$ (see Appendix \ref{sec:bai}). 

By combining these two algorithms and selecting the better of their outputs, we establish an overall approximation ratio of $O(n^{(r-1)/2})$ for the \tk SPM (\tkm SPM) problem on $\mathscr F_r$
(Theorem~\ref{thm:kspm}). Note that for the densest $k$-subhypergraph (D$k$SH) problem, the current best approximation ratios are $O(n^{1/4+\epsilon})$ for ordinary graphs ($r=2$) and $O(n^{0.697831+\epsilon})$ for 3-uniform hypergraphs (a special case of $r=3$). While our algorithm's approximation ratios do not surpass these when applied to these problems, its key advantage lies in its simplicity and its capability to handle all hypergraphs with rank $r$. 

Our work on greedy peeling for $k$SPM generalizes the current best $O(n/k)$-approximation ratio (with respect to $n$ and $k$ together) for the densest $k$-subgraph (D$k$S) problem \cite{kortsarz1993choosing,asahiro2000greedily}, which constitutes a special case of the $k$SPM with $r=2$. Notably, without the monotonicity assumption on the underlying supermodular function, a slight modification of the greedy peeling provides a \tkm-subset that is an $O(n^{r-1}/k^{r-1})$-approximation for the \tkm SPM problem (Corollary~\ref{cor:peeling}).

\paragraph{Algorithms for C\tkm SPM.} At a high level, the algorithms work by first generating a candidate pool of polynomially many connected \tkm-subsets; second, identifying and outputting the subset with the maximum value from this pool.

The second-tier strategy of the algorithm design employs a ``break-and-build'' approach. ``Breaking'' means fragmenting graph $G$ into ``densely'' connected components by eliminating elements with ``low'' utilities. The larger resulting components, which can deliver adequate value, become candidate \tkm-subsets added to the pool. ``Building'' operates in two ways: first, by grouping and merging smaller components scattered during the breaking phase to restore connectivity; second, by proactively generating candidate \tkm-subsets within connected sets using attachment methods. Specifically, to address the additional connectivity constraint, we use the idea of attaching ``valuable'' elements to a connected subset $S$ (referred to as the core). Based on the core and the input function $f$, we define a new supermodular function $f^S\in\mathscr F_{r-1}$ and construct an accompanying set $T$ that holds potential elements for attachment. All elements $u\in T$ with positive marginal values $f^S(u|T\setminus\{u\})$ are attached to the core to form a candidate $\underline{k}$-subset, where the positive marginals ensure not only the connectivity of the resulting \tkm-subset (Lemma~\ref{lem:connect}) but also, to some extent, its quality.

The third level of our algorithm design is the operational layer, which comprises two complementary algorithms. The first method, \emph{Star Exploration} (Algorithm 3), systematically explores star structures centered at each element (vertex) $v\in V$. For each core $\{v\}$,  batch-greedy augmenting (developed in Section \ref{sec:augment}) is used to construct $\{v\}$'s accompanying set $T$. Then, from $T$, promising ``leaves'' (elements with positive marginals) are selected and attached to the center $v$, forming a star-shaped connected \tkm-subset, which is added to the candidate pool. The best ``star'' in the pool yields an approximation ratio of $O(k^{r-1})$ for the C\tkm SPM problem (Theorem \ref{thm:star}).

The second method, \emph{Removing-Attaching-Merging} (Algorithm 4), is more intricate. This algorithm operates in stages, beginning with a \emph{removing} procedure (Procedure \ref{alg:pre}) that deletes elements with low marginal contributions (relative to density) and possibly other elements to produce a dense $\overline{k}$-subset $R$. The subsequent stage depends on the connectedness of $R$. 
\begin{itemize}
    \item  Either $R$ is connected: An \emph{attaching} procedure (Procedure \ref{alg:sunflower}) constructs a connected \tkm-subset of $R$ as a candidate subset: attaching elements (referred to as petals) to a connected $\lceil (r-1)k/r \rceil$-subset $C$ (the core). The petal selections are guided by function $f^C$ and $C$'s accompanying set $T$, which is constructed using greedy peeling (presented in Section \ref{sec:peel}). 
    \item 
Or $R$ is disconnected and contains no connected $\overline{k}$-subset:  All maximal connected $\overline{k/3}$-subsets of $R$ are added to the candidate pool; and a \emph{merging} procedure (Procedure~\ref{alg:merge}) groups and connects the remaining maximal connected subsets of $R$ (each with cardinality smaller than $k/3$) into an $O(n/k)$ number of larger candidate connected \tkm-subsets. The merging process constitutes the most technical part of our algorithm, where a spanning tree in a contracted graph is pruned step by step to show how these $\underline{(k/3-1)}$-subsets are merged under the principle that ``closed'' subsets have a higher chance of being merged.
\end{itemize}
By selecting the best connected \tkm-subset from the candidate pool, this composite algorithm provides an $O(n^{r-1}/k^{r-1})$-approximation for the C\tkm SPM problem (Theorem~\ref{thm:overall}).

Combining Algorithms \ref{alg:star} and \ref{alg:overall} (selecting the better of their outputs), we establish an overall approximation ratio of $O(n^{(r-1)/2})$ for the C\tkm SPM problem on $\mathscr F_r$ (Theorem \ref{thm:final}).

\bigskip

The remainder of this paper is organized as follows. Sections \ref{section:mkss} and \ref{section:mckss} present approximation algorithms for $k$SPM and C$\underline{k}$SPM problems, respectively. Section \ref{section:con} summarizes the paper and discusses possible directions for future research.

\section{Approximation under cardinality constraint}\label{section:mkss}
For any positive integer $h$, the symbol $[h]$ denotes the set of all positive integers at most $h$. Given a nonnegative supermodular function $f: 2^V \rightarrow \mathbb{R}_+$ and an integer $k\in[n]$, we are concerned with the ${k}$SPM  problem for finding a ${k}$-subset $S\subseteq V$ that maximizes $f(S)$. By iterative greedy peelings starting from $V$ (see Section \ref{sec:peel}) and iterative batch-greedy augmentations starting from $\emptyset$ (see Section \ref{sec:augment}), we derive a $k$-subset that is an $O(\min\{n^{r-1}/k^{r-1},k^{r-1}\})$-approximation for the \tk SPM problem on $\mathscr F_r$. To avoid discussing trivial cases, henceforth we assume $k\in[n-1]$.

\paragraph{Decomposition parameters.} Before proceeding with the algorithmic details, we discuss some basic properties of supermodular functions, which are useful in analyzing our algorithms. For any $S\subseteq V$, any normalized supermodular function $f$ satisfies the well-known property that 
\begin{equation}\label{eq:super}
  \sum_{v\in S}f(v|S\setminus\{v\})\ge f(S),  
\end{equation}
which follows from a telescoping sum. \citet{chekuri2022densest} defined the following parameter that reflects $f$'s quality in terms of modularity and decomposability: 
\[c_f:=\max_{S\subseteq V}\frac{\sum_{v\in S}f(v|S\setminus\{v\})}{f(S)},\]
where, for notational convenience, when $f(S)=0$,   it is assumed that $\frac{\sum_{v\in S}f(v|S\setminus\{v\})}{f(S)}=1$. It is straightforward that 
\begin{equation}\label{eq:cf}
 \min_{v\in S} f(v|S\setminus\{v\})\le \frac{\sum_{v\in S}f(v|S\setminus\{v\})}{|S|}\le c_f \frac{f(S)}{|S|},\text{ for all }S\subseteq V.   
\end{equation}
Moreover, $c_f$ serves as a lower bound for the decomposition property of the function $f$.
\begin{proposition}[\citet{chekuri2022densest}]\label{prop:cf}
    For any nonnegative $r$-decomposable supermodular function $f:2^V\rightarrow \mathbb{R}_+$, it holds that $c_f \le r$, and therefore $\sum_{v\in S}f(v|S\setminus\{v\})\le r\cdot f(S)$ for all $S\subseteq V$.
\end{proposition}

\paragraph{Normalization.} A function $f$ is \emph{normalized} if $f(\emptyset)=0$. Any function $f$ can be readily normalized by subtracting the constant $f(\emptyset)$ from its value for every subset. This normalization is a convenient preprocessing step that does not negatively impact our analysis. As shown in Appendix \ref{sec:norm}, the process does not decrease the approximation ratio for the supermodular maximization under consideration. Furthermore, it preserves essential properties of the function, including non-negativity, monotonicity, supermodularity (and $r$-decomposability). Henceforth, we assume that all such functions are normalized, even if not explicitly stated.

\subsection{Greedy peeling} \label{sec:peel}
Using a similar idea to \citet{kortsarz1993choosing} and \citet{asahiro2000greedily} for finding a dense subgraph, we give the following algorithm for $k$SPM, which repeatedly removes the element with the currently smallest marginal contribution. 
 
\smallskip

\begin{algorithm}[H]
    \caption{Greedy peeling for $k$-subsets}\label{alg:pmkss}
    \KwIn{Supermodular function $f:2^V\rightarrow \mathbb{R}_+$ and integer $k\in[n-1]$}
    $S_n\gets V$\;
    \For{$i\gets n$ \KwDownto $k+1$}
    {
        $v_i\gets\text{argmin}_{v\in S_i}\{f(v|S_i\setminus\{v\})\}$\;
        $S_{i-1}\gets S_i\setminus\{v_i\}$\;
    }
    \textbf{Output} $\textsc{Alg\ref{alg:pmkss}}(f,k):=S_k$.
\end{algorithm}

\medskip
 Note that the execution of Algorithm \ref{alg:pmkss} does not make any assumption on the decomposability of the underlying supermodular function. Even the value of the parameter $r$ does not need to be known.

\begin{lemma}\label{lma:pdk}
 If $f$ is $r$-decomposable and $r<k$, then   $f(S_i)/ f(S_j)\ge \binom{i}{r}/\binom{j}{r}\ge \frac{r!}{r^{r}}\cdot \frac{i^{r}}{j^{r}} $ for all $i,j$ with $k\le i\le j\le n$.
\end{lemma}
\begin{proof}
   For every integer $h\in[k+1,n]$, we have
    \[f(S_{h-1})=f(S_h)-\min_{v\in S_h} f(v|S_h\setminus\{v\}) \ge f(S_h)-c_f\frac{f(S_h)}{h}\ge f(S_h)-r\frac{f(S_h)}{h}=\frac{h-r}{h} f(S_h),\]
  where the first and second inequalities follow from (\ref{eq:cf}) and Proposition \ref{prop:cf}, respectively. Applying the inequality $f(S_{h-1})\ge\frac{h-r}{h} f(S_h)$ repeatedly yields
    \begin{align*}
        f(S_i)&\ge \frac{(j-r)\cdot (j-r-1)\cdots (i-r+1)}{j\cdot (j-1)\cdots (i+1)}\cdot f(S_j)\\
        &=\frac{(j-r)!/(i-r)!}{j!/i!}\cdot f(S_j)\\
        &=\frac{i!/(i-r)!}{j!/(j-r)!}\cdot f(S_j)
        \\
        &=\frac{ \binom{i}{r}}{\binom{j}{r}}
        \cdot f(S_j)
    \end{align*}
On the other hand, it follows from $i>r$ that $\frac{i(i-1)\cdots(i-r+1)}{r!}\ge(\frac{i}{r})^r$ and $\binom{i}{r}/\binom{j}{r}\ge  \frac{i\cdot (i-1)\cdots (i-r+1)}{j^{r}}\ge \frac{r!}{r^{r}}\cdot \frac{i^{r}}{j^{r}}$, proving the lemma.
\end{proof}

\begin{theorem}\label{thm:peeling}
    Algorithm \ref{alg:pmkss} runs in $O(n^2)$ time, and achieves an approximation ratio of  $O(n^{r-1}/k^{r-1})$ for the $k$SPM problem on $\mathscr F_r$, provided $k>r$.
\end{theorem}
\begin{proof} The running time is obvious. To prove the approximation ratio, let $f\in\mathscr F_r$ and $S_{opt}$ be an optimal solution of the problem and $\opt:=f(S_{opt})$. We construct a sequence of shrinking sets $S_k',S_{k-1}',\ldots,S'_0$ in this order. Initially set $S'_k:=S_{opt}$. Iteratively, for each $h=k,k-1,\ldots,1$, given $S'_h$, if there is $v'_h\in S_h'$ with $f(v_h'|S'_h\setminus\{v_h'\})\le \opt/(2k)$, then let $S'_{h-1} := S'_h \setminus\{v_h'\}$; otherwise, let $S_{h-1}' := S'_h$. So we have $S_{opt} = S'_k \supseteq S'_{k-1} \supseteq \cdots \supseteq S'_0$, $f(S'_0)\ge \opt-k\cdot \opt/(2k)=\opt/2$ and $f(v|S'_0\setminus\{v\})>\opt/(2k)$ for each $v\in S'_0$. 
    
    First, we consider the case that $f(v_i|S_i\setminus\{v_i\})\le \opt/(2k)$ for each $i\in [k+1,n]$. We claim that $S'_0\subseteq S_k$. Otherwise, there exists $i\in [k+1,n]$ such that $S'_0\subseteq S_i$ and $v_i\in S'_0$. However, since $f$ is supermodular, $f(v_i|S_i\setminus\{v_i\})\ge f(v_i|S'_0\setminus\{v_i\})>\opt/(2k)$, a contradiction. Moreover, since $f$ is monotone, we have
    \[\frac{\opt}{f(S_k)}\le \frac{\opt}{f(S'_0)}\le 2.\]

    Now, the remaining case is that there exists $t\in[k+1,n]$ such that $f(v|S_t\setminus\{v\})\ge \opt/(2k)$ for each $v\in S_t$. Then 
    by Lemma \ref{lma:pdk} and Proposition~\ref{prop:cf}, we have
    \begin{align*}
        f(S_k)&\ge \frac{r!}{r^{r}}\cdot \frac{k^{r}}{t^{r}}\cdot f(S_t)\\
        &\ge \frac{r!}{r^{r}}\cdot \frac{k^{r}}{t^{r}}\cdot \frac{\sum_{v\in S_t}f(v|S_t\setminus\{v\})}{r}\\
        &\ge \frac{r!}{r^{r}}\cdot \frac{k^{r}}{t^{r}}\cdot \frac{t\cdot \opt/(2k)}{r}
    \end{align*}
 and
    \[\frac{\opt}{f(S_k)}\le \frac{r^{r}\cdot 2r}{r!}\cdot \frac{t^{r-1}}{k^{r-1}} =O(n^{r-1}/k^{r-1}),\]
establishing the theorem.
\end{proof}

\subsection{Batch-greedy augmenting} \label{sec:augment}
Unlike its high efficiency in maximizing submodular functions, the classical greedy augmenting algorithm--starting from the empty set and greedily adding elements one by one--typically performs poorly for supermodular maximization (see Appendix~\ref{sec:bai} for examples). We show, however, that for $r$-decomposable functions, a batch-greedy variant that adds $r$ elements at a time (while respecting the cardinality constraint) achieves a provably good approximation.

\smallskip

\begin{algorithm}[H]
    \caption{Batch-greedy augmenting for $k$-subsets}\label{alg:chmkss2}
    \KwIn{Supermodular function $f:2^V\rightarrow\mathbb R_+$, integers $r$ and $k$ such that $1\le r\le k\le n-1$}
    $T_0\gets\emptyset$\;
    $t\gets\lfloor k/r\rfloor$\;
    \For{$i\gets1$ \KwTo $t$}{
        $\mathcal{R}_i\gets$ the set of $r$-subsets of $V\setminus T_{i-1}$\;
        $S_i\gets$ an $r$-subset in $\argmax_{S\in \mathcal{R}_i} f(T_{h-1}\cup S)$\;
        $T_i\gets T_{i-1}\cup S_i$\;
    }
    \textbf{Output}   $\textsc{Alg\ref{alg:chmkss2}}(f,r,k):=$ any $k$-subset of $V$ that contains $T_t$  
 
\end{algorithm}

\begin{lemma}\label{lma:chmkss2}
    Let $U_1,U_2,\dots,U_t$ be any $t$  pairwise disjoint $r$-subsets of $ V$. If $f$ is monotone, then $f(T_t)\ge \frac12\sum_{h=1}^{t} f(U_h)$.
\end{lemma}
\begin{proof} Suppose without loss of generality that $f(U_1)\ge \dots \ge f(U_t)$. We prove  $f(T_i)\ge \frac12\sum_{h=1}^{i} f(U_h)$ by induction on $i=1,2,\ldots,t$. The inequality trivially holds when $i=1$ because $U_1\in \mathcal{R}_1$ and $f(T_1)=f(S_1)\ge f(U_1)$. Proceeding inductively, we assume that $i\ge2$ and $f(T_{i-1})\ge \frac12\sum_{h=1}^{i-1} f(U_h)$. 

First we consider the case in which  $f(U_g\cap T_{i-1})\le f(U_g)/2$ for some $g\le i$. Clearly, there exists $S\in \mathcal{R}_i$ such that $U_g\subseteq T_{i-1}\cup S$. The supermodularity of $f$ implies $f(U_g\cup S)+f(T_{i-1})\le f(T_{i-1}\cup S)+f(T_{i-1}\cap U_g)$.  In turn, the definition of $T_i$ gives $f(T_i)-f(T_{i-1})\ge f(T_{i-1}\cup S)-f(T_{i-1})\ge f(U_g\cup S)-f(T_{i-1}\cap U_g)$. Since $f$ is monotone, we have
    \[f(T_i)-f(T_{i-1})\ge f(U_g)-f(T_{i-1}\cap U_g)\ge f(U_g)-f(U_g)/2=f(U_g)/2.\]
    By the induction hypothesis and $f(U_g)\ge f(U_i)$, we obtain
    \[f(T_i)= f(T_{i-1})+(f(T_i)-f(T_{i-1}))\ge \frac{\sum_{h=1}^{i-1} f(U_h)}{2}+\frac{f(U_g)}{2}\ge \frac{\sum_{h=1}^{i} f(U_h)}{2}.\] 
  
  It remains to consider the case in which $f(U_h\cap T_{i-1})\ge f(U_h)/2$ for all $h\le i$. Since $U_1,\ldots,U_i$ are mutually disjoint and $f(\emptyset)=0$, it follows from the supermodularity and monotonicity of $f$ that $ \sum_{h=1}^i f(U_h\cap T_{i-1})\le f((\cup_{h=1}^iU_h)\cap T_{i-1})\le f(T_{i-1})$. Furthermore,
    \[f(T_i)\ge f(T_{i-1})\ge \sum_{h=1}^i f(U_h\cap T_{i-1})\ge \frac{\sum_{h=1}^{i} f(U_h)}{2},\]
which completes the proof of the lemma.
\end{proof}

\begin{theorem}\label{thm:add}
    Algorithm \ref{alg:chmkss2} runs in  $O(n^{r+1})$ time, and achieves an approximation ratio of  $O(k^{r-1})$ for the $k$SPM  problem on $\mathscr F_r$, provided $k>r$.   
\end{theorem}
\begin{proof} Note that each $\mathcal R_i$ in the algorithm is of size $O(n^r)$ and can be constructed in $O(n^r)$ time. It follows that the algorithm runs in $O(n^{r+1})$ time. 

To prove the approximation ratio, suppose that the $r$-decomposition of $f\in\mathscr F_r$ is   $(V_i,f_i)_{i=1}^m$. (This decomposition is not an input to the algorithm; it is only used for the proof.)   Let $S_{opt}$ be an optimal solution of the problem and $\opt:=f(S_{opt})$. %Since $f$ is monotone, $S_{opt}$ is also an optimal solution of the \tkm SPM. 
For each $i=1,\ldots,m$, because $|V_i|\le r$,  we see that $V_i\cap S_{opt}$ is an $\underline{r}$-subset of both $V_i$ and $S_{opt}$, and therefore  %contained in 
the intersection of $V_i$ and some $r$-subset of $S_{opt}$.  Let $\mathcal R$ consist of all $r$-subsets of $S_{opt}$.  It follows from $f$'s decomposability and $f_i$'s nonnegativity %monotonicity 
that
    \[\opt=\sum_{i=1}^m f_i(V_i\cap S_{opt})\le \sum_{i=1}^m\sum_{S\in \mathcal{R}} f_i(V_i\cap S)=\sum_{S\in \mathcal{R}}\sum_{i=1}^m f_i(V_i\cap S)=\sum_{S\in \mathcal{R}}f(S).\]
  Let $\mathcal U$ be the set of vectors $(U_1,\ldots,U_t)$ with $U_1,\ldots,U_t$ being mutually disjoint $r$-subsets of $S_{opt}$. It is instant from Lemma \ref{lma:chmkss2} that $\sum_{(U_1,\dots,U_t)\in \mathcal{U} }\sum_{i=1}^t f(U_i)\le 2f(T_t)\cdot|\mathcal U|$. On the other hand, by symmetry, the number of appearances of an $r$-subset in $\mathcal R$ as an element of a vector in $\mathcal U$ is the same value $t|\mathcal U|/|\mathcal R|$, which gives $\sum_{(U_1,\dots,U_t)\in \mathcal{U} }\sum_{i=1}^t f(U_i)=(t|\mathcal{U}|/|\mathcal{R}|)\cdot \sum_{S\in \mathcal{R}}f(S)$. Hence
   $\sum_{S\in \mathcal{R}}f(S)\le 2f(T_t)|\mathcal{R}|/t=2f(T_t)\binom{k}{r}/t\le 2f(T_t)k^r/t$. Since $t= \lfloor k/r\rfloor\ge (k-r+1)/r$, we have
    \[\opt \le \sum_{S\in \mathcal{R}}f(S)\le 2r\cdot \frac{k^r}{k-r+1} \cdot f(T_t)=O(k^{r-1})\cdot f(T_t).\]
 \end{proof}

\subsection{Combining}
The combination of Theorems~\ref{thm:peeling} and \ref{thm:add} gives the following main result of this section. Taking the better solution output by Algorithms \ref{alg:pmkss} and \ref{alg:chmkss2} (when $k>r$), and using brute-force search to find the optimum (in case of $k\le r$), we can guarantee an approximation ratio of $O(\min\{k^{r-1},n^{r-1}/k^{r-1}\})$, where $\min\{k^{r-1},n^{r-1}/k^{r-1}\}\le n^{(r-1)/2}$. 
\begin{theorem}\label{thm:kspm}
    There is an  $O(n^{(r-1)/2})$-approximation algorithm for the $k$SPM and $\underline{k}$SPM problems on $\mathscr F_r$, which runs in $O(n^{r+1})$ time. 
\end{theorem}

\section{Approximation under cardinality and connectivity constraints}\label{section:mckss}
In this section, we design two approximation algorithms (Algorithm \ref{alg:star} and \ref{alg:overall}) for the connected \tkm-subset supermodular maximization (C\tkm SPM) problem on $\mathscr F_r$, achieving approximation ratios $O(k^{r-1})$ and $O(n^{r-1}/k^{r-1})$, respectively. Combining these two algorithms, we derive an $O(n^{(r-1)/2})$-approximation for the C\tkm SPM problem. These results extend to the C\tk SPM problem due to their ``indirect equivalence'', as discussed in Section \ref{sec:intr}.

Given function $f:2^V\rightarrow\mathbb R_+$ in $\mathscr F_r$, without loss of generality, we assume $f$ is normalized (see Appendix \ref{sec:norm}). Unlike the \tk SPM (\tkm SPM) counterpart, which only require the value of $r$ during our algorithmic solving process (i.e., the greedy augmenting subroutine), the C\tkm SPM setting explicitly provides the $r$-decomposition $(V_i,f_i)_{i=1}^m$ of $f$ as an input, where every $V_i$ is an $\underline{r}$-subset of $V$, every $f_i:2^{V_i}\rightarrow\mathbb{R}_+$ is supermodular, and the following \emph{decomposition identities} hold:
\begin{center}
 $f(S)=\sum_{i=1}^mf_i(S\cap V_i)$ for each $S\subseteq V$.
\end{center} 
For brevity, we denote the input of the C\tkm SPM problem by $f\in\mathscr F_{r}^V$, which specifies the ground set $V$ and implicitly provides the corresponding $r$-decomposition.

\paragraph{Underlying graph.} This decomposition is used to define the connectivity constraint of the C\tkm SPM problem. Specifically, $(V_i)_{i=1}^m$ yields an undirected graph $G=(V,E)$ with vertex set $V$ and edge set $E=\cup_{i=1}^m \{uv:\{u,v\}\subseteq V_i,u\neq v\}$. The objective of the C\tkm SPM is to find a connected \tkm-subset $S\subseteq V$ such that $f(S)$ is maximized, where ``connected'' means that $G[S]$, the subgraph of $G$ induced by $S$, is connected. 

We only need to consider the case where \emph{$V$ is connected}. This, in particular, implies that $r \ge 2$ and there is a $k$-subset that is optimal for the C\tkm SPM problem. %, as for $r = 1$, only singleton subsets are connected.  
If the graph $G$ is not connected, we can reduce the problem by considering each connected component separately: solving C$\underline{l}$SPM for the vertex set $U$ of each connected component of $G$, where $l=\min\{k,|U|\}$, and selecting the best solution as the overall output for the C\tkm SPM on $V$. Note that this reduction does not change the approximation ratio.
 
 Given that the ground set $V$ of the function $f$ coincides with the vertex set of the graph $G$, we shall \emph{use the terms ``element'' and ``vertex'' interchangeably}. 
 
 \paragraph{Decomposition constituents.} Notice from $0=f(\emptyset)=\sum_{i=1}^mf_i(\emptyset)$ that all $f_i:2^{V_i}\rightarrow\mathbb R_+$ are normalized. So $f_i(v|\emptyset)=f_i(v)\ge0$ holds for all $v\in V_i$. In turn, $f_i$'s supermodularity implies that it is monotone, since $f_i(v|S)\ge f_i(v|\emptyset)\ge0$ for any $S\subset V_i$ and any $v\in V_i\setminus S$. Thus
\begin{itemize}
    \item for each $i\in[m]$, $|V_i|\le r$ and  $f_i:2^{V_i}\rightarrow\mathbb{R}_+$ is nonnegative, supermodular, normalized and monotone.
\end{itemize}
 By $f$'s decomposability, we have $f(S)=\sum_{i=1}^mf_i(S\cap V_i)=f(S\cap (\cup_{i=1}^m V_i))$ for every $S\subseteq V$. Therefore, for notational convenience, we may assume without loss of generality that 
\begin{itemize}
    \item $V=\cup_{i=1}^m V_i$.
\end{itemize}

\paragraph{Algorithm overview.}
Before presenting the the technical details, we briefly outline our approach for finding a high-valued connected $\underline{k}$-subset. In Section~\ref{sec:attach}, we construct candidate \tkm-subsets by ``attaching'' elements to a common connected core, which takes one of two forms: either a single vertex serving as a star center (Section~\ref{sec:star}) or a connected  $\lceil (r-1)k/r \rceil$-subset functioning as a sunflower seed disc (Section~\ref{sec:sunflower}).
In Section~\ref{sec:remove}, we remove so-called removable elements from $V$ to ``lift'' both the density and marginal values of elements, and we further drop redundant elements whenever an entire maximal connected $\overline{k}$-subset survives. This process always maintains a set with at least $k$ elements, but its connectivity may be compromised. If the resulting set $R$ remains connected, then the sunflower attachment procedure, developed in Section~\ref{sec:sunflower}, applies to construct a \tkm-subset of $R$, yielding an $O(n^{r-1}/k^{r-1})$-approximation for the C\tkm SPM problem (see Lemma \ref{lem:rconnect}). If $R$ is not connected and contains no connected $\overline{k}$-subsets, we proceed to merging step in Section~\ref{sec:merge}. Here, we merge ``small'' maximal connected subsets of $R$ to form larger connected \tkm-subsets of $V$. We then select the subset with the highest value from among these newly formed subsets and existing ``big'' maximal connected subsets of $R$. Such a maximum-valued subset gives an $O(n/k)$-approximation of the C\tkm SPM problem (see Lemma \ref{lma:merge}).  
  Finally, in Section~\ref{integrate}, we integrate the removing, attaching, and merging procedures to derive an $O(n^{r-1}/k^{r-1})$-approximation in $O(mn^2)$ time. Combined with the $O(k^{r-1})$-approximation from the best star attachments (Section~\ref{sec:star}), this yields an overall $O(n^{(r-1)/2})$-approximation for the C\tkm SPM problem (see Theorem~\ref{thm:ckspm}).

\subsection{Concentric attachments}\label{sec:attach}
We say that two disjoint subsets of $V$ are \emph{connected to} each other if there exists a vertex in each of these two subsets such that the two vertices are neighbors in $G$. To ensure the connectivity of the \tkm-subset we construct, the most straightforward idea is to attach some elements to a connected subset at hand (referred to as the \emph{core}), where each element is connected to this core. In this way, the resulting \tkm-subset will be connected. For this purpose, we need to define a function  $f^S$ for the core $S$, and determine whether an element is connected to $S$ based on its marginal value with respect to some other set under $f^S$.

 For each nonempty $S\subseteq V$, let $I_S$ denote the set of indices $i\in[m]$ with $V_i\cap S\ne\emptyset$, and let function $f^S:2^{V\setminus S}\rightarrow \mathbb{R}$ be defined by 
\[f^S(T):=\sum_{\begin{subarray}{l}i\in I_S
\end{subarray}} \big(f_i((S\cup T)\cap V_i)-f_i(S\cap V_i)\big) \text{ for all }T\subseteq V\setminus S.\]
It is clear that $f^S$ is nonnegative and monotone. Note from $f$'s decomposition identities and $f_i$'s normality and nonnegativity that
\begin{gather}\label{eq:fs}
f^S(T)
= \left(\sum_{i\in I_S} f_i((S\cup T)\cap V_i)\right)-f(S)\le f(S\cup T)-f(S) \text{ for all }T\subseteq V\setminus S.\end{gather}
The leftmost equation in (\ref{eq:fs}) and the supermodularity of $f_i$ imply that $f^S$ is supermodular. For each $i\in I_S$, let $f^S_i:2^{V_i\setminus S}\rightarrow\mathbb R$ be defined by $f^S_i(T):=f_i((S\cup T)\cap V_i)-f_i(S\cap V_i)$ for all $T\subseteq V_i\setminus S$. Then we obtain the following decomposability for $f^S$.
\begin{lemma}\label{lem:fs}
    For any nonempty $S\subseteq V$, the function $f^S:2^{V\setminus S}\rightarrow \mathbb{R}_+$ is nonnegative, monotone, and supermodular. Moreover, it admits an $(r-1)$-decomposition $(V_i\setminus S,f^S_i)_{i\in I_S}$, where functions $f^S_i$, $i\in I_S$, are nonnegative, monotone and supermodular. \qed
\end{lemma}
%The intuition behind defining the function $f^{S}$ comes from the following observation. 
\begin{lemma}\label{lem:connect}
    For any $T\subseteq V\setminus S$ and   $u\in V\setminus S$, if $f^S(u|T)>0$, then   $\{u\}$ is connected to $S$.
\end{lemma}
\begin{proof}
   By the leftmost equation in (\ref{eq:fs}), we deduce from $f^S(u|T)>0$ that there must exist some $i\in I_S$ such that $( S\cup T\cup\{u\})\cap V_i$ \emph{properly} contains $(S\cup T)\cap V_i$, which particularly implies that $u\in V_i\setminus S$. By definition, $I_S\ne\emptyset$ gives $S\cap V_i\ne\emptyset$. In turn, the construction of $G$ says that vertex $u$ is a neighbor of every vertex in $S\cap V_i\ne\emptyset$. 
\end{proof}  

The lemma provides us a preliminary approach to ensure connectivity: starting from some connected subset $S$ (the core), seeking an appropriate \emph{accompanying set} $T\subseteq V\setminus S$, and picking the elements $u$ with positive marginal values $f^S(u|T\setminus\{u\})$, and attaching all these $u$ to core $S$ to form a connected \tkm-subset that possesses a large value under $f$. In the process, the selection of core $S$ and the construction of accompanying set $T$ are two crucial factors, which jointly determine the \emph{attachments} (those $u$ attached). 

In the next two subsections, we consider two cases for the core $S$. In the first case  (Section~\ref{sec:star}), we start with a $1$-subset (trivially connected) as the core (center) and expand it to a ``star''  by attaching ``leaves''. In the second case (Section~\ref{sec:sunflower}), we take a connected $\lceil(r-1)k/r\rceil$-subset as the core (seed disc) to form a ``sunflower'' by attaching ``petals''. The accompanying set $T$ is then constructed using the batch-greedy augmentation procedure (introduced in Section \ref{sec:augment}) for the star and the greedy peeling procedure (presented in Section \ref{sec:peel}) for the sunflower, with the leaf and petal attachments selected according to Lemma \ref{lem:connect}.

%To begin, we first examine the scenario where $S$ consists of a single vertex.

\subsubsection{Star exploration}\label{sec:star}
Using an idea similar to that in \citet{chen2017finding}, we derive the following algorithm for C$\underline{k}$SPM, which greedily explores stars centered at all vertices $v$ of graph $G$, and outputs the best one as an $O(k^{r-1})$-approximate solution (Theorem \ref{thm:chdcksh}). 

\smallskip

\begin{algorithm}[H]
    \caption{Star exploration for connected $\underline{k}$-subsets}\label{alg:chdcksh}\label{alg:star}
    \KwIn{$f\in\mathscr F_{r}^V$ and $k\in[n-1]$}
    \For{$v\in V    $}
    {
        $T^v\gets\text{\textsc{Alg}}2(f^{\{v\}},r-1,k-1)$\;
        $L^v\gets\{u\in T^v:f^{\{v\}}(u|T^v\setminus\{u\})>0\}$\;
        $S^v\gets L^v\cup \{v\}$\;
    }
    \textbf{Output} $\textsc{Alg\ref{alg:chdcksh}}(f,k):=$ any \tkm-subset in  $\argmax_{S^v} \{f(S^v)\}$
\end{algorithm}

\smallskip
The greedy nature here is reflected in the fact that the set $T$ (i.e., $T^v$ in the algorithm) accompanying the core 
$\{v\}$ is a $(k-1)$-subset constructed in Step 2 using the greedy augmenting (Algorithm~\ref{alg:chmkss2}), specifically for the $(r-1)$-decomposable supermodular function $f^{\{v\}}$.
In Step 3, we select all $u\in T^v$ that make positive marginal contribution to $T^v\setminus\{u\}$, where the $\underline{(k-1)}$-subset $L^v$ holds the vertices selected. By Lemma~\ref{lem:connect}, each vertex in $L^v$ is connected to $\{v\}$. Therefore, each \tkm-subset $S^v$ generated in Step 4 constitutes the vertex set of a star centered at $v$, which particularly shows that the \tkm-subset output by Algorithm~\ref{alg:star} is connected.
\begin{theorem}\label{thm:chdcksh}\label{thm:star}
    Algorithm \ref{alg:chdcksh} runs in $O(mn^{r+1})$ time and achieves an approximation ratio of $O(k^{r-1})$ for the C$\underline{k}$SPM  problem on $\mathscr F_r$, provided $k>r$. 
\end{theorem}
\begin{proof}
When calling the subroutine of Algorithm \ref{alg:chmkss2}, it necessitates invoking the oracle for $f^{\{v\}}$   $ O(n^{r+1})$ times (recalling Theorem~\ref{thm:add}). Furthermore, according to the definition of $f^{\{v\}}$, one run of its oracle requires $O(m)$ calls to the oracles of the $f_i$'s.  Consequently, the overall running time is upper bounded by $O(mn^{r+1})$. 

Let $S_{opt}^*$ be an optimal solution of the $\underline{k}$SPM and $\opt^*:=f(S^*_{opt})$. Then $\opt^*$ is an upper bound of the optimal objective value of the C$\underline{k}$SPM, and 
\[\sum_{v\in S^*_{opt}} f(v|S^*_{opt}\setminus\{v\})\ge\opt^*\] 
by $f$'s supermodularity (recall (\ref{eq:super})).  Notice from  %the proof of  
Theorem~\ref{thm:add} that there is a universal constant $\Phi$ such that  $f^{\{v\}}(S^*_{opt}\setminus\{v\})\le \Phi k^{r-2}\cdot f^{\{v\}}(T^v)$ holds for all $v\in S^*_{opt}$. For every $v\in S^*_{opt}$,  since $f^{\{v\}}$ is supermodular (by Lemma~\ref{lem:fs}), it is straightforward from its increasing marginal returns that $f^{\{v\}}(L^v)=f^{\{v\}}(T^v)$.  Recall from (\ref{eq:fs}) that $f(S^v)\ge f(v)+f^{\{v\}}(L^v)$.
In turn,   $f$'s nonnegativity gives
\[f(S^v)\ge f(v)+f^{\{v\}}(T^v)\ge \frac{f(v)+f^{\{v\}}(S^*_{opt}\setminus\{v\})}{\Phi k^{r-2}}\text{ for every }v\in S^*_{opt}.\]
Furthermore, since $f_i$'s are normalized, $f(v)=\sum_{i\in[m]}f_i(\{v\}\cap V_i)=\sum_{i\in[m]:v\in V_i}f_i(v)$. By the definition of $f^{\{v\}}$ and the nonnegativity of $f_i$'s, we have  
\begin{align*}
      f(v)+f^{\{v\}}(S^*_{opt}\setminus\{v\})&=\sum_{i\in[m]:v\in V_i}f_i(S^*_{opt}\cap V_i)\\
      &=\sum_{i=1}^mf_i(S^*_{opt}\cap V_i)-\sum_{i\in[m]:v\not\in V_i}f_i(S^*_{opt}\cap V_i)\\
      &\ge \sum_{i=1}^mf_i(S^*_{opt}\cap V_i)- \sum_{i=1}^mf_i((S^*_{opt}\setminus\{v\})\cap V_i)
 \end{align*} 
  It follows from $f$'s decomposition identities that 
  \[f(v)+f^{\{v\}}(S^*_{opt}\setminus\{v\})\ge f(S^*_{opt})-f(S^*_{opt}\setminus \{v\})=f(v|S^*_{opt}\setminus\{v\})\text{ for every }v\in S^*_{opt}.\]
    Therefore, the $\underline{k}$-subset output by Algorithm~\ref{alg:chdcksh} has value at least
    \[\max_{v\in V}f(S^v)\ge \frac{\sum_{v\in S^*_{opt}}f(S^v)}{|S^*_{opt}|}\ge \frac{\sum_{v\in S^*_{opt}} f(v|S^*_{opt}\setminus\{v\})}{k\cdot \Phi k^{r-2}  }\ge \frac{\opt^*}{\Phi k^{r-1}},\]
proving the $O(k^{r-1})$ approximation ratio.\end{proof}

Note that, in the above proof, the algorithm's performance is measured against the optimal solution $S^*_{opt}$ of the \tkm SPM problem, which does not impose the connectivity constraint.
\begin{corollary}\label{cor:star}
Provided that $k>r$ and $V$ is connected, Algorithm~\ref{alg:chdcksh}  outputs an $O(k^{r-1})$-approximate solution for the \tkm SPM problem on $\mathscr{F}_r$, which is connected.
\end{corollary}

In the star exploration of Algorithm~\ref{alg:chdcksh}, the initially connected set, comprising only a single center vertex, could potentially be too small for some instances. As a result, the leaf attachments might be disproportionately scattered, making it challenging to achieve a good balance between connectivity and the attained objective value. Conversely, if we start with a larger initial connected set $C$ --- the seed disc of a sunflower, then the selection of petal attachments according to $f^C$ may have more space to achieve a higher objective value.

\subsubsection{Sunflower growth} \label{sec:sunflower}
In this subsection, we take any connected  $\lceil (r-1)k/r \rceil$-subset $C$ as the sunflower disc (which, for example, can be constructed in $O(m)$ time by breadth-first search in $G$). Then, we use greedy peeling (introduced in Section~\ref{section:mkss}) to find the corresponding attachments, as specified in the following algorithm.

\begin{algprocedure}[H]
    \caption{Sunflower growth for connected \tkm-subsets}\label{alg:amckss}\label{alg:sunflower}
   \KwIn{$f\in\mathscr F_r^V$ and $k\in[n-1]$}
    $C \gets$ a connected $\lceil (r-1)k/r \rceil$-subset of $V$\;
    $T\gets$  \textsc{Alg}\ref{alg:pmkss}$(f^C,\lfloor k/r \rfloor)$\;\label{stp:t}
    $P \gets \{u\in T:f^C(u|T\setminus\{u\})>0\}$\;
    \textbf{Output} $\textsc{Prc\ref{alg:amckss}}(f,k):=C\cup P$.
\end{algprocedure}
Although the procedure indeed outputs a feasible solution for C\tkm SPM (see Lemma~\ref{lma:amckss}), its approximation ratio against the optimum can be arbitrarily bad in the worst case. On the other hand, its performance could be quantified in terms of  a universal lower bound $\delta$ on element marginal values and a positive constant $c$, which are defined as follows: 
\[\delta:=\min_{v\in V}f(v|V\setminus\{v\})\text{ and }c:=\frac{(r-1)!}{r^r\cdot (2r-2)^{r-1}}.\]
\begin{lemma}\label{lma:amckss}
 Procedure \ref{alg:amckss} runs in $O(mn^2)$   time, and outputs  a connected \tk-subset $S:=C\cup P$  of $V$ such that $f(S)\ge c\cdot (k/n)^{r-1} \cdot k\delta$.
\end{lemma}

\begin{proof}
 The runtime is dominated by the computation at Step~\ref{stp:t}. By Theorem~\ref{thm:peeling}, it involves $O(n^2)$ calls of $f^C$'s oracle, each of which consists of $O(m)$ calls of $f_i$'s oracles. 
 
 Note that $f^C$  is defined on $2^{V\setminus C}$, giving that $T$ is a $\lfloor k/r \rfloor$-subset of $ V\setminus C$. For each $u\in P$, since $f^{C}(u|T\setminus\{u\})>0$, by Lemma~\ref{lem:connect}, $\{u\}$ is connected to $C$. It follows from $C$'s connectivity that 
  $S$ is a connected \tk-subset of $V$. 
  
  It is straightforward  that $f^{C}(T)=f^{C}(P)$. By Lemma \ref{lma:pdk} and the $(r-1)$-decomposability of $f^C$ (recalling Lemma~\ref{lem:fs}), 
    we have
        \begin{align*}
            f^C(T)&\ge \left(\frac{(r-1)!}{(r-1)^{r-1}}\cdot \left\lfloor\frac{k}{r}\right\rfloor^{r-1}\middle/\left(n-\left\lceil \frac{(r-1)k}{r} \right\rceil\right)^{r-1}\right) \cdot f^C(V\setminus C)\\
            &\ge \frac{(r-1)!}{(r-1)^{r-1}}\cdot \left(\frac{k}{2r}\right)^{r-1} \cdot \left(\frac{1}{n}\right)^{r-1} f^C(V\setminus C)\\
            &= \frac{(r-1)!}{(2r-2)^{r-1}}\cdot \left(\frac{k}{rn}\right)^{r-1} f^C(V\setminus C),
        \end{align*}
  where the second inequality follows from $k/r-\lfloor k/r\rfloor<1\le\lfloor k/r\rfloor$.  By $f$'s decomposition identities, we have
    \[f(v|V\setminus\{v\})=f(V)-f(V\setminus\{v\})=\sum_{i: v\in V_i} f_i(V_i)-f_i(V_i\setminus\{v\})\le \sum_{i: v\in V_i} f_i(V_i),\]
    and 
    \begin{align*}
        f^{C}(V\setminus C)&=\left(\sum_{i: V_i\cap C\neq \emptyset}f_i(V_i)\right)-f(C)\\
        &= \left(\sum_{v\in C}\sum_{i: v\in V_i}\frac{f_i(V_i)}{|V_i\cap C|}\right)-f(C)\\
        &\ge \sum_{v\in C}\left(\sum_{i: v\in V_i}\frac{f_i(V_i)}{r-1}-\sum_{\substack{i: v\in V_i\subseteq C}}\frac{f_i(V_i)}{r(r-1)}\right)-f(C)\\
        &\ge \sum_{v\in C}\frac{f(v|V\setminus\{v\})}{r-1}-\sum_{i:V_i\subseteq C}\sum_{v\in V_i}\frac{f_i(V_i)}{r(r-1)}-f(C)\\
        &\ge \delta\cdot \frac{(r-1)k}{r}\cdot \frac{1}{r-1}-\sum_{i:V_i\subseteq C}\frac{|V_i|\cdot f_i(C)}{r(r-1)}-f(C)\\
        &\ge \frac{\delta k}{r} -\frac{f(C)}{r-1}-f(C)\\
        &=\frac{\delta k}{r}-\frac{rf(C)}{r-1},
    \end{align*}
where the first equation is implied by $f_i$'s normality, and the third inequality is implied by their monotonicity. Recall from (\ref{eq:fs}) that $f(S)=f(C\cup P)\ge f(C)+f^C(P)$. Thus,
    \begin{align*}
        f(S)&\ge f(C)+f^C(T)\\
        &\ge f(C)+\frac{(r-1)!}{(2r-2)^{r-1}}\cdot \left(\frac{k}{rn}\right)^{r-1} \cdot \left(\frac{k\delta}{r}-\frac{rf(C)}{r-1}\right)\\
        &= \frac{(r-1)!}{(2r-2)^{r-1}}\cdot\frac{1}{r^{r}}\left(\frac{k}{n}\right)^{r-1} \cdot k\delta+\left(1-\frac{r!}{(r-1)^r\cdot (2r)^{r-1}}\cdot\left(\frac{k}{n}\right)^{r-1}\right)f(C)\\
        &\ge c\cdot (k/n)^{r-1} \cdot k\delta,
    \end{align*}
  establishing the lemma.
\end{proof}

\subsection{Element removal}\label{sec:remove}
Recall from Theorem \ref{thm:dkpss} that one can find in $O(n^2)$ time a $\overline{k}$-subset  $\bv$ of $V$ that is an $(r+1)$-approximation of the densest $\overline{k}$-subset, i.e.,  $(r+1)\cdot\frac{f(\bv)}{|\bv|}\ge \frac{f(S)}{|S|}$ for every $\overline{k}$-subset $S$ of $V$. Since $f\in\mathscr F_r$ and $V$ is connected, the C\tkm SPM admits a $k$-subset as an optimal solution. Let $\opt$ denote the optimal objective value of the C\tkm SPM. Then
\begin{equation} \label{eq:dense}
   \frac{f(\bv)}{|\bv|}\ge \frac{\opt}{k(r+1)}.
\end{equation} 

For any $S\subseteq V$, an element $v\in S$ is called \emph{removable} in $S$ if $f(S\setminus\{v\})/(|S|-1)>f(S)/|S|$. 

\begin{algprocedure}[H]
    \caption{Element removals from $\overline{k}$-subsets}\label{alg:pre}
     \KwIn{$f\in\mathscr F_r^V$ and $k\in[n-1]$}
    $T\gets$ a $\overline{k}$-subset $\bar V$ of $V$ that satisfies (\ref{eq:dense})\;
    \While{$|T|>k$ and $T$ has a removable element}{
       $v\gets$ a removable element of $T$\;
       $T\gets T\setminus\{v\}$\;
    }
    $R\gets T$\;\label{step:outputb}
    \If{$|T|>k$ and $T$ contains a connected $\overline{k}$-subset}{
        $R\gets$ a maximal connected $\overline{k}$-subset of $T$\;\label{step:outputa}
       }
    \textbf{Output} $\textsc{Prc\ref{alg:pre}}(f,k):=R$ 
\end{algprocedure}
Henceforth, we reserve symbol $R$ to represent the output of Procedure \ref{alg:pre}, which takes $f$, its $r$-decomposition and $k$ as input. 

Moreover, the restriction of $f$ to $2^R$, denoted as $f|_{2^R}$, admits an $r$-decomposition $(V_i\cap R,f_{i}|_{2^{V_i\cap R}})_{i\in I_R}$, where $I_R$ consists the indices $i\in[m]$ with $V_i\cap R\ne\emptyset$. Thus, we can write $f|_{2^R}\in \mathscr F_r^R$ for brevity.

\begin{lemma}\label{lma:pre}
    Procedure \ref{alg:pre} runs in $O(n^2)$ time and outputs a $\overline{k}$-subset $R$ of $V$. One of the following holds:
    \begin{enumerate}[(a)]
        \item The procedure skip Step \ref{step:outputa}, and $f(R)\ge \opt/(r+1)$. \label{lma:prea}
        \item The procedure executes Step \ref{step:outputa},  $R$ is connected, $f(v|R\setminus\{v\})\ge \opt/(k(r+1))$ for each $v\in R$ and  $f(R)\ge \opt/(r(r+1))$. \label{lma:preb}
    \end{enumerate}
\end{lemma}
\begin{proof}
Clearly, $|R|\ge k$, and the runtime is dominated by the computation of $\bar V$, which takes $O(n^2)$ time. 

By the definition of removal elements, when the procedure terminates, the final $T$ satisfies $f(T)/|T|\ge f(\bar V)/|\bar V|$. Furthermore inequality (\ref{eq:dense}) gives \[\frac{f(T)}{|T|}\ge \frac{\opt}{k(r+1)}.\] 
    The validity of \ref{lma:prea} is straightforward, given that the set $R$ is assigned its final setting at Step~\ref{step:outputb} of the procedure. 
    
    To prove \ref{lma:preb}, we assume that the procedure executes Step \ref{step:outputa}. Then the final $T$, which consists of more than $k$ elements, cannot have any removable element. Thus, by definition, for each $v\in T$, we have $f(T\setminus\{v\})\le(|T|-1)f(T)/|T|$, and
    \[f(v|T\setminus\{v\})=f(T)-f(T\setminus\{v\})\ge f(T)-f(T)\cdot \frac{|T|-1}{|T|}=\frac{f(T)}{|T|}\ge \frac{\opt}{k(r+1)}.\]
   Observe that $R$ induces a connected component in $G[T]$. So $f(v|R\setminus\{v\})=f(v|T\setminus\{v\})\ge \opt/(k(r+1))$ for each $v\in R$. Besides, since $f$ is $r$-decomposable, it follows from Proposition~\ref{prop:cf} that 
    \[rf(R)\ge \sum_{v\in R} f(v|R\setminus\{v\})\ge \sum_{v\in R} \frac{\opt}{k(r+1)}=|R|\cdot \frac{\opt}{k(r+1)}\ge  \frac{\opt}{r+1}.\]
    It leads to $f(R)\ge \opt/(r(r+1))$ as desired.
\end{proof}

\begin{lemma}\label{lem:rconnect}
    If $R=\textsc{Prc\ref{alg:pre}}(f,k)$ is connected, then either $R$ is an $r(r+1)$-approximation for the C\tkm SPM, or the output by Procedure \ref{alg:amckss}, when given $f|_{2^R}\in\mathscr F_r^R$ and $k\in[n-1]$ as input, is an $O(n^{r-1}/k^{r-1})$-approximation for the  C\tkm SPM.
\end{lemma}
\begin{proof}
By Lemma \ref{lma:pre}, we know that $|R|\ge k$ and $f(R)\ge {\opt}/{(r(r+1))}$.
If $R$ is a $k$-subset, then it is a feasible solution of the C\tkm SPM achieving the approximation ratio $r(r+1)$. 

 It remains to consider the case where $|R| > k$.  Suppose that $S:=\textsc{Prc\ref{alg:amckss}}(f|_{2^R},k)$. Then by Lemma \ref{lma:amckss} and Lemma \ref{lma:pre}(b), we see that $S$ is a connected \tkm-subset of $R$ satisfying 
    \[f(S)\ge c\cdot \frac{k^r}{|R|^{r-1}} \cdot \min_{v\in R}f(v|R\setminus\{v\}) \ge c\cdot \frac{k^r}{n^{r-1}} \cdot \frac{\opt}{k(r+1)}
    \ge \frac{c}{r+1}\cdot \left(\frac{k}{n}\right)^{r-1} \cdot \opt,\]
   which proves the lemma.
\end{proof}

The preceding lemma establishes an $O(n^{r-1}/k^{r-1})$-approximation for the  C\tkm SPM instances with $f\in\mathscr F_r^V$, when $R$ is connected. The next subsection addresses the complementary case in which $R$ is not connected. In this setting, we still have $|R|\ge k$ (by Lemma \ref{lma:pre}), but (by the execution of Procedure \ref{alg:pre})  no connected subset of $R$ contains $k$ or more elements. $R$ is the disjoint union of its maximal connected subsets, each of which is often referred to as a \emph{connected component} of $R$.

\subsection{Component merging}\label{sec:merge}
    In this subsection, we assume that the set $R$ returned by Procedure \ref{alg:pre} is not connected, and each connected component of $R$ contains no more than $k-1$ elements. 
    
  As established in Lemma~\ref{lma:pre},  $f(R)$ has been a constant approximation of the optimal objective value $\opt$:
 \begin{equation*}
     f(R)\ge\frac{\opt}{r+1}. 
 \end{equation*}
  % The supermodularity of $f$ implies that 
  Due to the function's decomposable structure, $f(R)$ equals the sum of the function values $f(C)$ over the connected components $C$ of $R$. If $R$ consists of a small number of connected components, then the component with the highest value serves as a good approximation for the C\tkm SPM problem. The difficulty arises when $R$ is partitioned into a large number of connected components, most of which have a small cardinality.
  \begin{itemize}
        \item A connected component $S$ of $R$ is called \emph{small} if $|S|<k/3$, and \emph{big} otherwise. 
        \item Let $\mathbb S$ (resp.\ $\mathbb{B}$) denote the set of all small (resp.\ big) connected components of $R$. We often call each member of $\mathbb S$ a \emph{small component} for short.
  \end{itemize}
 \begin{equation}
    \sum_{S\in\mathbb{S}\cup\mathbb{B}}f(S)=f(R)\ge\frac{\opt}{r+1}.\label{eq:fr}
 \end{equation}
    Our basic idea is to group the small components in $\mathbb S$. For each group, we add some additional elements to connect them, forming a single connected subset. On one hand, we need to ensure that the number of groups is not too large, thereby overcoming the aforementioned difficulty. On the other hand, we must guarantee that each of these connected subsets is a \tkm-subset.
    
    The technical challenge lies in how to perform the grouping. To control the number of groups, a natural idea is to ensure that the total cardinality of the members (components in $\mathbb{S}$) in each group is not too small. However, considering this alone is insufficient. This is because connecting the group members requires adding extra elements, and the increase in cardinality from these additions depends on the graphical structure. For example, components that are ``far apart'' would require a large number of additional elements to connect. If they were grouped together, their original total cardinality must not be too large (due to the cardinality constraint $k$). In order to group small components in $\mathbb{S}$ that are ``close'' to each other as much as possible, we need to utilize the positional information of these small connected components within the graph $G$. To do this, we shrink each small component in $\mathbb S$ into a single node and find a spanning tree in the resulting graph. Using the hierarchical structure and relative positions of these nodes in the spanning tree, we can group and connect (merge) the small connected components that are relatively close to each other, thereby keeping both the number of groups and the cardinality of the connected subsets within appropriate ranges. This enables us to obtain a set $\mathbb M$ of $O(n/k)$ connected \tkm-subsets that collectively cover $R$ (see Procedure~\ref{alg:merge}). By choosing the best connected \tkm-subset in $\mathbb{M}\cup \mathbb{B}$, we derive an $O(n/k)$-approximation for the C\tkm SPM in the case of disconnected $R$ (see Theorem~\ref{lma:merge}).

\paragraph{Shrinking.} For convenience, we often use $|G|$ to denote the number of vertices in the graph $G=(V,E)$. Given any connected subset $S$ of $V$, \emph{shrinking} $S$ in $G$ means gluing all vertices in $S$ together to form a new vertex $v_S$ such that all vertices of $V \setminus S$ and all edges in $E$ with both ends in $V \setminus S$ are unchanged, each edge joining a vertex $u \in V \setminus S$ and a vertex in $S$ becomes an edge joining $u$ and $v_S $, and all edges with both ends in $S$ are deleted. We say that $v_S$ \emph{represents} (the vertices of) $S$.
\begin{itemize}
      \item Let $G/\mathbb S$ be the graph derived from $G$ by shrinking all members of $\mathbb S$, where each vertex outside  small components in $\mathbb{S}$ \emph{represents} itself.
\end{itemize}

To distinguish between the vertices in $G$ and those in $G/\mathbb S$, we refer to the latter as ``nodes.''
For any subgraph $H$ of $G/\mathbb{S}$, let $\Lambda(H)$ denote the set of vertices in $V$ that are represented by the nodes in $H$. %The size of $H$ is defined as $\lambda(H)=|\Lambda(H)|$.

\paragraph{BFS spanning tree.} Starting from an arbitrary node $t$ in $G/\mathbb S$, a breadth first search (BFS) over $G$ constructs a spanning tree $T$ of $G/\mathbb S$ \emph{rooted} at $t$, along with a \emph{reverse BFS ordering} $v_1,v_2,\ldots,v_{|T|}$ of its nodes, such that
\begin{itemize}
 \item  $v_{|T|}$ is the root $t$, and each node $v_i$ $ (1\le i\le|T|)$ has its index $i$ lower than its parent. 
\end{itemize}
For any subtree $T'$ of $T$ that contains $v_{|T|}$, and any node $v_i$ in $T'$, we also considered $T'$ being rooted at $v_{|T|}$, and  use $T'_i$ to denote the (maximal) subtree of $T'$ rooted at $v_i$, consisting of $v_i$ and all its descendants in $T'$. It is clear that $T'\setminus T'_i$ is a subtree of $T$ that contains $v_{|T|}$ unless $i=|T|$ (in which case, $T'\setminus T'_{|T|}=T'\setminus T':=\emptyset$). 

\paragraph{Merging by tree pruning.} The for-loop (Steps \ref{stp:for} --\ref{stp:endfor}) of Procedure~\ref{alg:merge} below merges small components to form the set $\mathbb M$, via pruning the spanning tree $T$ in a bottom-up fashion (from leaves to root). Each time a subtree is removed or several branches are pruned, the removed part corresponds to a connected \tkm-subset, which is added to $\mathbb M$. The for-loop iterates through the nodes of the tree in order $v_1, v_2, \dots, v_{|T|}$, considering the subtree rooted at each $v_i$ in turn. For each such subtree, the procedure works as follows:
\begin{itemize}
    \item  If the cardinality of the vertex set represented by the current subtree is less than $k/3$, the procedure does nothing and proceeds to the next node $v_{i+1}$. 
    \item If the cardinality falls within $[k/3, k]$, the procedure removes the subtree rooted at $v_i$ (Step \ref{step:delete2}) and places the vertex set it represents into $\mathbb M$ (Step \ref{stp:add-1}).
\item If the cardinality exceeds $k$, the procedure identifies a set of subtrees rooted at some children of $v_i$ such that the total number of vertices they represent is between $k/3$ and $2k/3$ (Step \ref{step:c}). The procedure then adds the set composed of these vertices, along with the vertices represented by $v_i$, to $\mathbb M$ (Step \ref{stp:add-more}), and removes these child subtrees (Step \ref{step:delete1}; note that $v_i$ itself is not removed at this stage). This process of pruning subtrees rooted at some children of $v_i$ and expanding $\mathbb M$ is repeated until the number of vertices represented by the subtree rooted at $v_i$ is at most $k$.
\end{itemize}
After handling the subtree rooted at $v_i$  in this way, the procedure continues to  $v_{i+1}$ and repeats the above steps.

\smallskip

\begin{algprocedure}[H]
    \caption{Small components merging for connected \tkm-subsets}\label{alg:merge}
    \KwIn{$\overline{k}$-subset $R=\textsc{Prc\ref{alg:pre}}(f,k)$ that is not connected}
    $\mathbb S\gets$ the set of small connected components of $R$\;
     $\mathbb{B}\gets$ the set of connected components of $R$ that are not small\;
    %Let $G/\mathbb{S}$ be obtained from $G$ by shrinking all elements of $\mathbb{M}$\;
    $\mathbb{M}\gets\emptyset$\;
    \If{$\mathbb S\ne\emptyset$}{
     $T\gets$ a spanning tree of $G/\mathbb{S}$ with reverse BFS ordering $v_1,\dots,v_{|T|}$ of its nodes\;\label{stp:tree}
        \For{$i=1$ \KwTo $|T|$}{\label{stp:for}
        \While{$|\Lambda(T_i)|> k$}{\label{step:condition}
            $C\gets$ a set of some children of $v_i$ in $T_i$ such that $k/3\le \sum_{v_h\in C} |\Lambda(T_h)|\le 2k/3$\;\label{step:c} \tcp{See the proof of Lemma~\ref{lma:merge} for the construction of $C$.}
            $\mathbb{M}\gets\mathbb{M}\cup \{\Lambda(v_i)\cup (\bigcup_{v_h\in C} \Lambda(T_h))\}$\;\label{stp:add-more}
            $T\gets T\setminus (\bigcup_{v_h\in C} T_h)$\; \label{step:delete1}
        }
        \If{$k/3\le |\Lambda(T_i)|\le k$}{
            $\mathbb{M}\gets\mathbb{M}\cup \{\Lambda(T_i)\}$\;\label{stp:add-1}
            $T\gets T\setminus T_i$\; \label{step:delete2}
        }
        
    }
    \If{$T\ne\emptyset$}{
        $\mathbb{M}\gets\mathbb{M}\cup \{\Lambda(T)\}$\;\label{stp:endfor}
    }}
    \textbf{Output} $\textsc{Prc\ref{alg:merge}}(R):=$ any set in $\argmax_{S'\in \mathbb{M}\cup \mathbb{B}} f(S')$
\end{algprocedure}

\begin{lemma}\label{lma:merge}
   Given disconnected $R=\textsc{Prc\ref{alg:pre}}(f,k)$, Procedure \ref{alg:merge} runs in $O(m)$ time and satisfies the following statements:
    \begin{enumerate}[(i)]
        \item The procedure produces a set $\mathbb M$ of connected \tkm-subsets such that $|\mathbb{M}|\le 3n/k+1$ and $\sum_{S\in \mathbb{M}} f(S)\ge \sum_{S\in \mathbb{S}} f(S)$. 
        \item The procedure outputs a connected \tkm-subset $M=\textsc{Prc\ref{alg:merge}}(R)$ such that $(r+1)\cdot (6n/k+1)\cdot f(M)\ge\opt$.
    \end{enumerate}
\end{lemma}

\begin{proof}
   To establish the correctness and runtime of  Procedure \ref{alg:merge}, let $d_i$ denote the degree of node $v_i$ in the initial spanning tree computed at Step~\ref{stp:tree}. We observe that each $v_i$ can only be eliminated from the tree $T$ at or after the end of $i$-th iteration of the for-loop. We prove by induction on $i$ the following claim for the $i$-th iteration of the for-loop:
    \begin{itemize}
        \item[(\textsc{c1})] when $|\Lambda(T_i)|> k$,  the children set $C$ as specified in Step~\ref{step:c} can be found in $O(d_i)$ time; and
        \item[(\textsc{c2})]  at the end of the iteration, either $v_i$ is eliminated from $T$, or $|\Lambda(T_i)|<k/3$.
    \end{itemize}
    The base case corresponds to the leaves of the spanning tree and holds true, because, by the definition of $\mathbb{S}$, we have $|\Lambda(v)|<k/3$ for each node $v\in G/\mathbb S$. 
    
    Suppose the claim (consisting of (\textsc{c1}) and (\textsc{c2})) holds for all indices smaller than $i$. If $|\Lambda(T_i)|<k/3$ at the start of the $i$th iteration, then the iteration has done nothing and the claim is true. If $k/3\le |\Lambda(T_i)|\le k$ at the start of the $i$-th iteration, then the while-loop has done nothing and $v_i$ is deleted from the tree at the end of the $i$-th iteration (Step \ref{step:delete2}). Now we consider the cases that $|\Lambda(T_i)|> k$ at the start of the $i$th iteration, for which the while-loop is executed.  Let $v_{i_1},\dots,v_{i_c}$ be all the children of $v_i$ at that time. By the property of the reverse BFS ordering, the indices $i_1,\ldots,i_c$ of these children are all smaller than $i$. It follows from the induction hypothesis that 
    \begin{center}
        $|\Lambda(T_{i_j})|<k/3$ for all $j\in[c]$.
    \end{center}     
    Since $|\Lambda(T_i)|>k$ and $|\Lambda(v_i)|<k/3$, we have $\sum_{j=1}^c |\Lambda(T_{i_j})|=|\Lambda(T_i)|-|\Lambda(v_i)|> 2k/3$. Let $c'$ be the smallest index such that $\sum_{j=1}^{c'} |\Lambda(T_{i_j})|\ge 2k/3$. Then $|\Lambda(T_{i_{c'}})|<k/3$ enforces $k/3<\sum_{j=1}^{c'-1} |\Lambda(T_{i_j})|\le 2k/3$. So we can set $C=\{v_{i_1},\dots,v_{i_{c'-1}}\}$, which, along with $c<d_i$, justifies (\textsc{c1}). Consider the last execution of the while-loop in the $i$-th iteration of the for-loop. At the beginning of the execution, Step \ref{step:condition}, the set $\Lambda(T_i)$ has cardinality $|\Lambda(T_i)|> k$. After removing a $\overline{2k/3}$-subset $\bigcup_{v_h\in C} \Lambda(T_h)$ at Step \ref{step:delete1}, it becomes a \tkm-subset, from which we deduce that $k/3\le |\Lambda(T_i)|\le k$ holds at the end of the while-loop. In turn, Step \ref{step:delete2} eliminates $T_i$, which includes $v_i$. This proves (\textsc{c2}) and, consequently, the claim, thereby confirming the correctness of Procedure~\ref{alg:merge}. The runtime follows from the fact that $\sum^{|G/\mathbb{S}|}_{i=1}d_i=2(|G/\mathbb{S}|-1)\le 2n$ and it takes $O(m)$ time for Step~\ref{stp:tree} to construct a BFS spanning tree in $G/\mathbb{S}$.

    By the claim, it is straightforward to check that each subset we put into $\mathbb M$ is a connected \tkm-subset. Recall that initially $|\Lambda(T)| = n$. Each time we put a set in $\mathbb M$, we must immediately delete a rooted subtree from $T$ (at Step \ref{step:delete2}) or delete several subtrees rooted at some children of the same node from $T$ (at Step \ref{step:delete1}). In either case, $|\Lambda(T)|$ decreases by at least $k/3$. It follows that  $|\mathbb{M}|\le 3n/k+1$.

   By the construction of $G/\mathbb S$, each small connected component in $\mathbb S$ is contained in some connected \tkm-subset in  $ \mathbb{M}$. Then, since these small connected components are pairwise disjoint, we have
    \[\sum_{S\in \mathbb{M}} f(S) \ge \sum_{S\in \mathbb{M}} f\bigg(\bigcup_{S'\in\mathbb S:S'\subseteq S}S'\bigg)\ge \sum_{S\in \mathbb{M}} \sum_{S'\in\mathbb S:S'\subseteq S} f(S')\ge\sum_{S'\in \mathbb{S}} f(S'),\]
where the first inequality is implied by $f$'s monotonicity, and the second is by $f$'s supermodularity and normality. Statement (i) is proved.

To see statement (ii), we recall that the members of $\mathbb{B}$ are pairwise disjoint $\overline{k/3}$-subsets. So $|\mathbb{B}|\le n/(k/3)=3n/k$, and by (i),  $S$ is a connected \tkm-subset of $V$.  Moreover, recalling (\ref{eq:fr}), we have
    \[\sum_{S\in  \mathbb{M}\cup \mathbb{B}} f(S)\ge \sum_{S\in  \mathbb{S}\cup\mathbb{B}} f(S)\ge \frac{\opt}{r+1}.\]
    It implies that 
    \[\opt\le (r+1)\sum_{S\in \mathbb{M}\cup\mathbb{B}}f(S)\le(r+1)(|\mathbb{M}|+|\mathbb{B}|)\cdot \max_{S\in \mathbb{M}\cup\mathbb{B}}f(S)\le (r+1)\cdot (6n/k+1)\cdot f(M),\]
 establish (ii).   \end{proof}

\subsection{Integration}\label{integrate}
In this subsection, we first integrate Procedures  \ref{alg:amckss} -- \ref{alg:merge} into an algorithm (Algorithm~\ref{alg:mckss}) for solving the C\tkm SPM problem, achieving an approximation ratio of  $O(n^{r-1}/k^{r-1})$. Then, we combine this algorithm with the star-exploration algorithm (Algorithm \ref{alg:chdcksh}), which has an approximation ratio $O(k^{r-1})$, to provide an $O(n^{(r-1)/2})$-approximation for the C\tkm SPM problem.

In Algorithm \ref{alg:mckss} below, we first run Procedure \ref{alg:pre}, removing removable elements one by one from $V$ until no removable elements remain or only $k$ elements are left. The resulting set $R$ is a $\overline{k}$-subset. If $R$ is a connected $k$-subset (Step~\ref{stp:kset}), it is returned as the output of the algorithm. If $R$ is connected but contains more than $k$ elements (Step~\ref{stp:else1}), we run Procedure  \ref{alg:amckss}. Starting from any connected $\lceil (r-1)k/r \rceil$-subset of $R$, the ``sunflower seed disc'', we grow it into a connected \tkm-subset, ``sunflower'', through petal attachments, and this subset is returned as the output.
If $R$ is not connected (Step \ref{stp:else}), we run Procedure \ref{alg:merge} to merge small connected components of $R$, forming a ``small'' number of connected \tkm-subsets. Among these subsets and the big connected components of $R$, the best subset is selected as the algorithm's output.

\smallskip

\begin{algorithm}[H]
    \caption{Removing-Attaching-Merging for connected \tkm-subsets}\label{alg:mckss}\label{alg:overall}
     \KwIn{$f\in\mathscr F_r^V$ and $k\in[n-1]$}
    $R\gets\textsc{Prc}\ref{alg:pre}(f,k)$\;
    \If{$R$ is connected}{
        \If{$|R|=k$}{\label{stp:kset}
              $S\gets R$\;\label{step:output1}
        }
        \Else{\label{stp:else1}
             $S\gets \textsc{Prc\ref{alg:amckss}}(f|_{2^R},k)$\; \label{step:output2}
        }
    }
    \Else{\label{stp:else}
        $S\gets$   \textsc{Prc}\ref{alg:merge}($R$)\;
           }
            \textbf{Output} $S$ \label{step:output3}
\end{algorithm}

\begin{theorem}\label{thm:mckss}\label{thm:overall}
   Algorithm \ref{alg:mckss} runs in $O(mn^2)$ time and achieves an approximation ratio of  $O(n^{r-1}/k^{r-1})$ for the C$\underline{k}$SPM  problem on $\mathscr F_r$, provided $k>r$. 
\end{theorem}
\begin{proof}
    The runtimes of Procedures \ref{alg:amckss}, \ref{alg:pre} and \ref{alg:merge} are $O(mn^2)$, $O(n^2)$, and $O(m)$, respectively, as specified in Lemmas \ref{lma:amckss}, \ref{lma:pre} and \ref{lma:merge}. Thus, the overall runtime of the algorithm is $O(mn^2)$. The approximation ratio $O(\max\{r(r+1),n^{r-1}/k^{r-1},(r+1)(6n/k+1)\})$, i.e.,  $O(n^{r-1}/k^{r-1})$, follows immediately from Lemmas \ref{lem:rconnect} and \ref{lma:merge}.
\end{proof}

Similar to the analysis of Algorithm \ref{alg:star}, when evaluating the performance of Algorithm \ref{alg:overall}, we can replace the optimal value $\opt$ for the C\tkm SPM problem in inequality (4) with the optimal value $\opt^*$ for the \tkm SPM problem, and hence replace every occurrence of $\opt$ in the subsequent analysis with $\opt^*$. This implies that, if $V$ is connected, then Algorithm \ref{alg:overall} in fact outputs an $O(n^{r-1}/k^{r-1})$-approximation for \tkm SPM, which is connected.
\begin{corollary}\label{cor:overall}
Provided that $k>r$ and $V$ is connected, Algorithm~\ref{alg:mckss}  outputs an $O(n^{r-1}/k^{r-1})$-approximate solution for the \tkm SPM problem on $\mathscr{F}_r$, which is connected.
\end{corollary}

The combination of Theorems~\ref{thm:chdcksh} and \ref{thm:mckss} gives the following main result of this section. Taking the better solution output by Algorithms \ref{alg:chdcksh} and \ref{alg:mckss} (when $k>r$), and using brute-force search to find the optimum (in case of $k\le r$), we can guarantee an approximation ratio of $O(\min\{k^{r-1},n^{r-1}/k^{r-1}\})$ with $\min\{k^{r-1},n^{r-1}/k^{r-1}\}\le n^{(r-1)/2}$. 

\begin{theorem}\label{thm:ckspm}\label{thm:final}
    There is an   $O(n^{(r-1)/2})$-approximation algorithm for the C$\underline{k}$SPM (resp.\ C\tk SPM) problem on $\mathscr F_r$, which runs in $O(mn^{r+1})$ time. %nonnegative supermodular functions that are monotone and $r$-decomposable, provided $r$ is a constant.
\end{theorem}

Combined with Corollaries \ref{cor:star} and \ref{cor:overall}, we obtain the following corollary.

\begin{corollary}\label{cor:strong}
 Suppose that $V$ is connected. Then there is a polynomial-time algorithm that outputs an $O(n^{(r-1)/2})$-approximate solution for the $\underline{k}$SPM (resp.\ \tk SPM) problem on $\mathscr F_r$, which is connected.
\end{corollary}

Taking $r=2$ in the above theorem, our algorithm improves the approximation ratio for the densest connected $k$-subgraph (DC$k$S) problem from the previously best $O(n^{2/3})$ \cite{chen2017finding} to $O(n^{1/2})$.
\begin{corollary}
   There is an $O(\sqrt{n})$-approximation algorithm for the  DC$k$S  problem, which runs in $O(mn^3)$ time. 
\end{corollary}

When $r=2$, the performance ratio $O(\sqrt{n})$ in Corollary \ref{cor:strong} is best possible, because, as shown for the DC$k$S problem \cite{chen2017finding}, the ratio between the maximum density of $k$-subgraph and the maximum density of a connected $k$-subgraph can be as high as $\Omega(\sqrt{n})$.

\section{Concluding remarks}\label{section:con}
In this paper, we study the cardinality-constrained supermodular maximization problems, addressing versions both with and without a connectivity requirement.   We present polynomial-time approximation algorithms for both sets of problems, assuming the underlying nonnegative supermodular function is monotone and $r$-decomposable, with $r$ being a constant. Our algorithms yield two complementary approximation ratios: $O(k^{r-1})$ and  $O(n^{r-1}/k^{r-1})$, which combine to achieve an overall  $O(n^{(r-1)/2})$-approximation ratio.

\subsection{Approximation for \tkm SPM on non-monotone functions}
Although for monotone functions, we can solve the \tkm SPM problem using the algorithm for the $k$SPM problem, this approach generally does not work for non-monotone functions. However, a slight modification to the greedy peeling algorithm, Algorithm  \ref{alg:pmkss} (originally designed for $k$SPM), enables us to handle the \tkm SPM problem for non-monotone functions. Let Algorithm  \ref{alg:pmkss}' denote the modification of Algorithm  \ref{alg:pmkss}, where Step 5 is replaced with ``\textbf{\textup{Output}} $S_k$ or $\emptyset$ whichever has the larger function value''.

\begin{corollary}\label{cor:peeling}
    Algorithm \ref{alg:pmkss}'  achieves an approximation ratio of  $O(n^{r-1}/k^{r-1})$ for the  $\underline{k}$SPM  problem on nonnegative supermodular functions that are $r$-decomposable, provided $r<k$.
\end{corollary}

Indeed, in the proof of Theorem~\ref{thm:peeling}, the monotonicity of the function $f$ is used only in the first case, where $f(v_i|S_i\setminus\{v_i\})\le \opt/(2k)$ for each $i\in[k+1,n]$. Because $f$ is nonnegative and supermodular,  
    \[f(A)\le f(A)+f(B\setminus A)\le f(B)+f(\emptyset)\le 2\max\{f(B),f(\emptyset)\}\text{ holds for all }A\subseteq B\subseteq V.\] 
    Thus, for the case, we now have \[\frac{\opt}{\max\{f(S_k),f(\emptyset)\}}\le \frac{2\cdot \opt}{f(S'_0)}\le 4,\]
  as a substitute for the conclusion $\opt/f(S_k)\le2$ drawn for the case in monotone setting. The remaining arguments in the proof of Theorem \ref{thm:peeling} remain valid for non-monotone functions, and the $O(n^{r-1}/k^{r-1})$ approximation ratio in Corollary~\ref{cor:peeling} follows.  

\subsection{Optimization under tight connectivity}
Beyond the standard approach to defining connectedness of subsets (which we have discussed so far), there is another natural, yet stronger, notion of tight connectivity. Given $r$-decomposition $(V_i,f_i)_{i=1}^m$ of function $f\in\mathscr F_r$, we call a subset $S$ of $V$ {\em tightly connected} if the hypergraph with vertex set $S$ and hyperedge set $\{V_i:S\subseteq V_i,i\in[m]\}$ is connected. Clearly, if $S$ is tightly connected, then $S$ is connected; however, the converse is not necessarily true. Moreover, for the ground set $V$, connectedness and tight connectedness coincide.

We prove in Appendix \ref{sec:scon} that the supermodular maximization for finding a maximum-valued tightly connected $\underline{k}$-subset preserves the asymptotic approximation ratio achievable for the \tkm SPM problem. In particular, the connectivity property stated in Corollary \ref{cor:strong} is strengthened to tight connectivity, as highlighted by the following corollary. 
\begin{corollary}\label{cor:tight-strong}
  Suppose that $V$ is connected. Then there is a polynomial-time algorithm that outputs an $O(n^{(r-1)/2})$-approximate solution for the $\underline{k}$SPM problem on $\mathscr F_r$, which is tightly connected.
\end{corollary}

\subsection{Future research directions} 
This work has focused on maximizing over supermodular functions with cardinality and connectivity constraints. Richer constraints open several intriguing research directions. Two of the most natural extensions are knapsack and matroid constraints.

\textbf{Knapsack constraints.} In many practical applications, such as resource allocation and budgeting problems, a simple cardinality constraint may be too  {restrictive}. A more generalized form of the problem is to consider a knapsack constraint: assign a nonnegative weight $w(v)$ to each element $v \in V$, and require that the weighted sum of the subset $S \subseteq V$, $\sum_{v \in S} w(v) $, does not exceed a given budget $W$. %, i.e., $\sum_{v \in S} w(v) \leq W$. 
Note that the cardinality constraint is a special case where $W = k$ and $w(v) = 1$ for all $v \in V$. \citet{sviridenko2004note} has studied the problem of submodular maximization under knapsack constraints.
    
\textbf{Matroid constraints.} Matroids provide a rich generalization of independence structures and have been widely used in combinatorial optimization. Matroid constraints often arise in network design and machine learning applications, where dependencies between elements need to be captured. Given a matroid $(V, \mathcal I)$ consisting of a finite set $V$ and a non-empty collection $\mathcal I$ of subsets of $V$. The matroid constraint restricts feasible solutions to members of $\mathcal I$. Note that the cardinality constraint corresponds to the uniform matroid, i.e., $\mathcal I = \{ S \subseteq V : |S| \leq k \}$. \citet{calinescu2011maximizing} has studied the problem of submodular maximization under matroid constraints.
 
Exploring supermodular maximization under these richer constraint classes promises both theoretical challenges and practical impact.

% \bibliographystyle{elsarticle-num-names-alpha}
% \bibliography{ref}

\newpage
\appendix

\section{The arbitrarily large approximation ratio of the single-element augmentation for 2-decomposable functions} \label{sec:bai} %\texorpdfstring{\citet{bai18greed}}{Bai}
The algorithm proposed by \citet{bai18greed} starts from the empty set and greedily adds a single element with the highest marginal value at each step (see Algorithm \ref{alg:bai}).

\begin{algorithm}[H]
    \caption{Greedy algorithm for $k$SPM by \citet{bai18greed}}\label{alg:bai}
    \KwIn{A supermodular function $f:2^V\rightarrow \mathbb{R}_+$ and an integer $k\in[n]$}
    $S_0\gets\emptyset$\;
    \For{$i=1$ \KwTo $k$}
    {
        $s_i\gets\text{any element} \in \text{argmax}_{v\in V\setminus S_{i-1}}\{f(v|S_{i-1})\}$\;
        $S_{i}\gets S_{i-1}\cup \{s_i\}$\;
    }
    \textbf{Output} $S_k$
\end{algorithm}

The approximation ratio of this algorithm for monotone 2-decomposable nonnegative supermodular functions can be arbitrarily large, even for a ground set of fixed size $n \geq 4$. To see this, consider the function $f:2^V \rightarrow \mathbb R_+$ defined on $V=\{v_1,v_2,\dots,v_n\}$ as follows 
\[f(S)=M\cdot n_1 (S)\cdot(n_1 (S)-1)+n_2(S) \text{ for all }S\subseteq V,\]
where $M\gg 1$,  $n_1(S)=|S\cap \{v_1,v_2\}|$, $n_2(S)=|S\cap \{v_3,v_4,\dots,v_n\}|$. It is routine to check that $f\in\mathscr F_2$, 
where its 2-decomposition is $(2;(\{v_1,v_2\},f_1),(\{v_{i+1}\},f_i)_{i=2}^{n-1})$ with $f_1(S)=M|S|(|S|-1)$ for all $S\subseteq\{v_1,v_2\}$ and $f_i(S)=|S|$ for all $i\in \{2,\dots,n-1\}$ and $S\subseteq\{v_{i+1}\}$. 

For each $2\le k\le n-2$, the optimal solution of the $k$SPM is $S^*=\{v_1,v_2,\dots,v_k\}$ with $f(S^* )=2M+k-2$. However, since $f(v_1|S)=f(v_2|S)=0$ for any $S\subseteq V$ with $S\cap \{v_1,v_2\}=\emptyset$, the algorithm never adds $v_1$ or $v_2$ to $S_i$ during its execution. Consequently,  the solution $S_k$ obtained by this algorithm satisfies $S_k\cap \{v_1,v_2\}=\emptyset$, resulting in $f(S_k) = k$ and an arbitrarily large approximation ratio $\Theta(M/k)$.

\section{Normalization assumption} \label{sec:norm}
The following proposition shows that normalizing a nonnegative set function $f$ does not decrease the approximation ratio for the maximization problem $\max_{S\in\mathcal F}f(S)$, and preserves the function's property of being nonnegative, monotone, supermodular (and $r$-decomposable).
\begin{proposition}\label{prop:norm}
    Given nonnegative function $f:2^V\rightarrow\mathbb R_+$, let $g:2^V\leftarrow\mathbb R$ be the normalized function defined by $g(S)=f(S)-f(\emptyset)$ for all $S\subseteq V$. Then the following hold.
    \begin{enumerate}[(i)]
        \item   ${f(T)}/{f(S)}\le {g(T)}/{g(S)}$ for all $S,T\subseteq V$ with $f(S)\le f(T)$.
        \item  $g$ is nonnegative and monotone if $f$ is monotone.
     \item  $g$ is supermodular if $f$ is supermodular.
   \item  $g$ is $r$-decomposable if $f$ is monotone, supermodular and $r$-decomposable.
    \end{enumerate}
\end{proposition}
\begin{proof}
   It suffices to consider the case of $f(\emptyset)> 0$.
  The inequality in (i) follows from $g(T)f(S)-f(T)g(S)=(f(T)-f(\emptyset))f(S)-f(T)(f(S)-f(\emptyset))=f(\emptyset)\cdot (f(T)-f(S))\ge 0$. Statements (ii) and (iii) are trivial.
    
    To prove (iv), we assume that $V=[n]$ and $f$ admits $r$-decomposition $(V_i,f_i)_{i=1}^m$. Let $I_j=\{i\in[m]:j\in V_i\}$. Consider any $j\in V$. Since $f$ is monotone,  we have $f(j|\emptyset)=\sum_{i\in I_j} f_i(j|\emptyset)\ge 0$. Therefore, for each $i\in I_j$, there exists $w_{ij}\ge 0$ such that \[f(j|\emptyset)=\sum_{i\in I_j} w_{ij}.\]
    To be specific, when $i\in I_j$, %$I_j\neq \emptyset$, 
    we can set $w_{ij}$ to $f(j|\emptyset)$ if  $i=\min I_j$, and  to $0$ otherwise. Let $g_i:2^{V_i} \rightarrow \mathbb{R}_+$ be the function defined as follows: 
    \[g_i(S)=f_i(S)-f_i(\emptyset)+\sum_{j\in S} (w_{ij}-f_i(j|\emptyset))\text{ for each }S\subseteq V_i.\]
    We claim that $(V_i,g_i)_{i=1}^m$ is an $r$-decomposition of $g$. First, since $f_i$ is supermodular, 
    \[g_i(S)\ge \sum_{j\in S} f_i(j|\emptyset)+\sum_{j\in S} (w_{ij}-f_i(j|\emptyset))=\sum_{j\in S} w_{ij}\ge 0\text{ for each }S\subseteq V_i,\]
    saying that $g_i$ is nonnegative. Second, it is easy to see that 
    \[g_i(A\cup B)+g_i(A\cap B)-g_i(A)-g_i(B)=f_i(A\cup B)+f_i(A\cap B)-f_i(A)-f_i(B)\ge 0\text{ for all }A,B\subseteq V_i,\]
   showing that $g_i$ is supermodular. Finally, $f$'s decomposability implies that for each $S\subseteq V_i$ 
    \begin{align*}
        \sum_{i=1}^m g_i(S \cap V_i)&=\sum_{i=1}^m \left(f_i(S \cap V_i)-f_i(\emptyset)+\sum_{j\in S \cap V_i} (w_{ij}-f_i(j|\emptyset))\right)\\
        &=f(S)-f(\emptyset)+\sum_{j\in S}\sum_{i\in I_j} (w_{ij}-f_i(j|\emptyset))\\
        &=g(S)
    \end{align*}
  which justifies (iv).
\end{proof}

\section{Tight connectivity requirements}\label{sec:scon}
Given function $f\in\mathscr F_r$ with its $r$-decomposition $(V_i,f_i)_{i=1}^m$, a subset $S$ of $V$ is tightly connected if and only if the graph with vertex set $S$ and edges set $\cup_{i\in[m]:V_i\subseteq S} \{uv:\{u,v\}\subseteq V_i,u\neq v\}$ is connected. 

The {tightly connected \tkm-subset supermodular maximization} problem is to find a tightly connected \tkm-subset $S$ of $V$ with maximum $f(S)$. If $k< r^2+r$, then we can use a brute-force search to find the optimal tightly connected \tkm-subset. For $k\ge r^2+r$, the following theorem implies that the supermodular maximization on $\mathscr F_r$ subject to cardinality and tight connectivity constraints admits the same approximation ratio (up to a constant) as one derives for the \tkm SPM problem.

\begin{theorem}
Suppose that $k\ge r^2+r$.  Let $S_k^*$ be the optimal solution of the \tkm SPM problem on $f\in\mathscr F_r$. If there is a polynomial algorithm that finds a connected \tkm-subset $S_k$ with $f(S_k)\ge \alpha \cdot f(S_k^*)$, then there is a constant $c$ and a polynomial algorithm that finds a tightly connected \tkm-subset $S_k'$ satisfying $f(S_k')\ge c\alpha \cdot f(S_k^*)$.
\end{theorem}
\begin{proof}
     Let $k'=\lfloor k/r\rfloor$ and $S^*_{k'}$ be any $\underline{k'}$-subset in $\mathrm{argmax}\{f(S):|S|\le k', S\subseteq S_k\}$. By Corollary \ref{cor:overall}, %Theorem \ref{thm:mckss}, 
     we can find a connected $\underline{k'}$-subset $S_{k'}\subseteq S_k$ in polynomial time such that $f(S_{k'})\ge c_1\cdot f(S_{k'}^*)$, where $c_1$ is a constant. By Lemma \ref{lma:pdk}, there is a constant $c_2$ such that $f(S_{k'}^*)\ge c_2 \cdot f(S_k)$. 
    
    Recall that $G=(V,E)$ is a graph with edge set $E=\cup_{i=1}^m \{uv:\{u,v\}\subseteq V_i,u\neq v\}$. Since $S_{k'}$ is connected, we have a spanning tree $T$ of the graph $G[S_{k'}]$. By definition, we can construct a function $g:E(T)\rightarrow [m]$ satisfying $e\subseteq V_{g(e)}$ for each $e\in E(T)$. Define subset $S'_k:=S_{k'}\cup (\bigcup_{e\in E(T)} V_{g(e)})$. Then $|S'_k|\le |S_{k'}|+(|S_{k'}|-1)(r-2)\le r|S_{k'}|\le k$ and $S'_k$ is tightly connected. Besides,  the monotonicity of $f$ implies
    \[f(S'_k)\ge f(S_{k'})\ge c_1\cdot f(S_{k'}^*)\ge c_1c_2 \cdot f(S_k)\ge c_1c_2\alpha \cdot f(S_k^*),\]
verifying the theorem.
\end{proof}

\end{document}